\documentclass[letterpaper]{siamonline1116}

\usepackage{amsfonts}
\usepackage{graphicx}
\usepackage{epstopdf}
\usepackage{algorithmic}
\ifpdf
  \DeclareGraphicsExtensions{.eps,.pdf,.png,.jpg}
\else
  \DeclareGraphicsExtensions{.eps}
\fi
\usepackage[latin1]{inputenc}
\usepackage[english]{babel}
\usepackage{listings,relsize}
\usepackage{amsmath}
\usepackage{amssymb}
\usepackage{eucal}
\usepackage{mathrsfs}
\usepackage[hang,small,bf]{caption}
\usepackage[labelformat=simple]{subcaption}

\usepackage{bm}
\usepackage{url}
\usepackage{booktabs}
\usepackage{hyperref}
\usepackage{harvard}
\usepackage{siunitx}
\usepackage{alltt}
\usepackage{color}
\usepackage{fullpage} 
\usepackage{geometry}
\usepackage{cite}
\definecolor{string}{rgb}{0.7,0.0,0.0}
\definecolor{comment}{rgb}{0.13,0.54,0.13}
\usepackage[draft]{todonotes}  		

\numberwithin{theorem}{section}

\newcommand{\TheTitle}{Canards in stiction: on solutions of a friction oscillator  by  regularization}

\title{{\TheTitle}\thanks{Submitted to the editors March 13, 2017.
\funding{ }}}

\author{
  Elena Bossolini\thanks{Department of Applied Mathematics and Computer Science, Technical University of Denmark,
Kongens Lyngby 2800, DK
    (\email{ebos@dtu.dk},\email{mobr@dtu.dk},\email{krkri@dtu.dk}).}
  \and
  Morten Br\o ns\footnotemark[2]  \and
  Kristian Uldall Kristiansen\footnotemark[2]
}

\usepackage{amsopn}

\graphicspath{{ArXiv_Figures/}}
\newcommand{\vare}{\varepsilon}
\newcommand{\be}{\begin{equation}}
\newcommand{\ee}{\end{equation}}

\DeclareMathOperator{\sign}{sign}
\hypersetup{colorlinks=true,linkcolor=blue,pdfauthor= Elena Bossolini}

\newsiamremark{remark}{Remark}  
\crefname{remark}{Remark}{Remarks}  

\begin{document}
\maketitle

\begin{abstract}
We study the solutions of a friction oscillator subject to stiction. This discontinuous model is non-Filippov, and the concept of Filippov solution cannot be used. Furthermore some Carath\'eodory solutions  are unphysical. Therefore we introduce the  concept of stiction solutions: these are the Carath\'eodory solutions that are physically relevant, i.e. the ones that follow the stiction law.  However, we find that some of the stiction solutions  are forward non-unique  in subregions of the slip onset. We call these solutions singular, in contrast to the regular stiction solutions that are forward unique. In order to further the understanding of the non-unique dynamics, we introduce a regularization of the model. This gives a singularly perturbed problem that captures the main features of the original discontinuous problem. We identify a repelling slow manifold that separates the forward slipping to forward sticking solutions, leading to a high sensitivity to the initial conditions. On this slow manifold we find canard trajectories, that have the physical interpretation of delaying the slip onset. We show with numerics that the regularized problem has a family of periodic orbits interacting with the canards.  We observe that this family has a saddle stability and that it  connects, in the rigid body limit, the two  regular, slip-stick branches   of the discontinuous problem, that were otherwise disconnected.  
\end{abstract}



\section{Introduction}
Friction is a tangential reaction force that appears whenever two rough surfaces are in  contact.  This energy-dissipating force is desirable  in car brakes \cite{cantoni2009a}, it occurs at the boundaries  of the Earth's crustal plates during fault  slip \cite{Nakatani2001,woodhouse2015a}, and it  causes the sound of string instruments \cite{akay2002a,feeny1998a}. Friction may  initiate undesirable noise, like the squeaking  of the chalk on a blackboard, or the squealing of  train wheels in tight curves \cite{heckl2000a}. It may also induce chattering vibrations, as   in  machine tools \cite{pratt1981a}, and in relay feedback systems \cite{olsson2001a}.\\ 
The variety  of examples above-mentioned underlines the importance of understanding the friction force, although this is far from being accomplished.  For instance, little is known on the  shape of the friction law for small  velocities, 
 as it is difficult to verify it experimentally \cite{putelat2010a,hinrichs1996a}. Yet, it is recognized that the maximal value of the friction force at {\it stick}, that means at zero relative velocity, is higher than  at {\it slip}, when the two surfaces are in relative motion \cite{rabinowicz1951a}.    Several models of friction exist in the literature\cite{olsson1997a,pennestri2016a,wojewoda2008a,woodhouse2015a}, and most of them are discontinuous  at stick, like the   stiction model. 
 Stiction   defines  a maximum {\it static} friction force during {stick} and a lower,  {\it dynamic} friction force at  slip. In subsets of the discontinuity, the stiction model has solutions that are forward non-unique.  In these subsets,  a numerical simulation requires a choice of forward integration, possibly discarding   solutions. \\
This manuscript aims to unveil, through a mathematical analysis,  new features of the stiction law around the {\it slip onset},  i.e. when the surfaces start to slip.  The manuscript shows that, in certain circumstances,  the slip onset is delayed with respect to the instant where the external forces have equalled the maximum static friction. 
This result,  that in principle  could be tested experimentally, has physical implications that may further the understanding of phenomena related to friction.\\ 
The paper studies the new features of the stiction law   in a model of a friction oscillator subject to stiction  \cite{shaw1986a}. This   is a discontinuous system, and one may attempt to study it by using the well-developed theory of Filippov \cite{filippov1988a,dibernardo2008a}. However, it turns out that the model is non-Filippov, and therefore the concept of Filippov solution cannot be used. New concepts of solution of a discontinuous system are introduced, but they lack  forward uniqueness in certain subregions of  the slip onset. Here it is not possible to predict whether the oscillator will  slip or  stick in forward time. To deal with the  non-uniqueness, a regularization  is introduced \cite{sotomayor1996a,kristiansen2014a}: this gives a smooth, singularly perturbed problem, that captures the main features of the original problem.  Singular perturbation methods \cite{jones1995a} can be used to study the regularized system.  The lack of uniqueness turns into a high sensitivity to the initial conditions, where a repelling  slow manifold   separates sticking from slipping solutions. Along this manifold    canard-like trajectories appear.  These canard trajectories are the ones that delay the slip onset. \\
 It is already known that the friction oscillator may exhibit chaotic \cite{licsko2014a,hinrichs1998a} and  periodic behaviour \cite{csernak2006a,olsson2001a,popp1990a}. The manuscript shows, with a numerical computation,  that there exist a family of slip-stick periodic orbits  interacting  with the canard solutions.  This family  connects, at the rigid body limit,  the two branches of slip-stick orbits of the discontinuous problem.  Furthermore the orbits of this family are highly unstable, due to an ``explosion'' of the Floquet multipliers.  \\
The manuscript is structured as follows. 
\Cref{sec:model} presents the model and \cref{sec:PWS_analysis} studies its geometrical structure.  \Cref{sec:solution_definition} introduces a concept of solution that makes sense for the discontinuous model and \cref{sec:regularization} introduces the regularization.   \Cref{sec:slipstickorbits}  shows  slip-stick periodic orbits  interacting with the canard solutions.  Finally \cref{sec:conclusion} concludes the manuscript and discusses the results. 

\section{Model}\label{sec:model}
A friction oscillator consists of a mass $M$  that sits on a rough table, as shown in \cref{fig:model}, and that is subject to a periodic forcing $F_{\omega}(\bar{t}) := -A \sin(\omega \bar{t})$, with $A$   and $\omega$  parameters and $\bar{t}$ time. 
The mass is connected to   a spring of stiffness $\kappa$, that at rest has zero length. Hence the spring elongation $u$ corresponds to the position of $M$.  Besides, the motion of the mass on the rough table  generates a frictional force $F$ that  aims to oppose this movement. 
\begin{figure}[!t]
	\centering
	\includegraphics[width=0.45\textwidth]{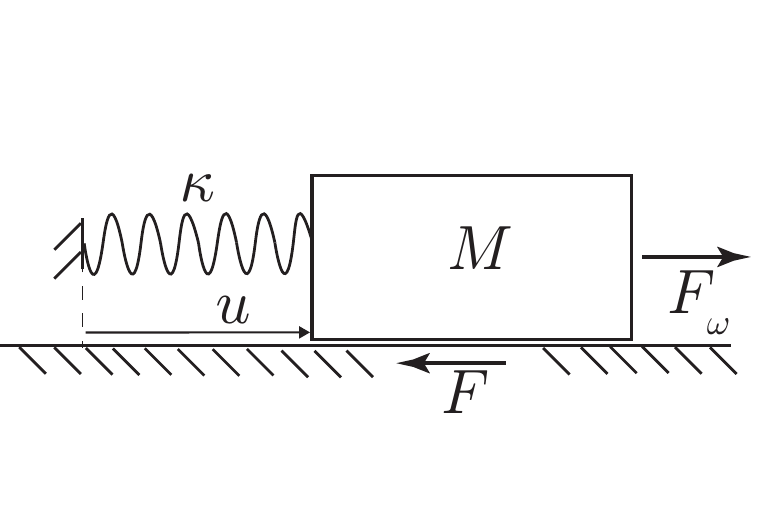}
	\caption{Model of a friction oscillator.}
	\label{fig:model}
\end{figure}
The system of  equations describing the friction oscillator is
\be \label{eq:friction_oscillator}
\begin{aligned}
\dot{u} &= v,\\
M\dot{v} &= -  \kappa u + F_\omega(\bar{t})+ F.
\end{aligned}\ee
The friction force $F$   is modelled as stiction. According to this law, $F$ has different values depending on whether the slip velocity  $v$ is zero or not.  During slip ($v\neq0$) stiction is identical to  the classical Coulomb law: the friction force is constant and acts in the opposite direction of the relative motion,
\be \label{eq:slip_law}
F = -N f_d \sign v \quad \text{when} \quad v\neq 0.
\ee
In equation \cref{eq:slip_law} the parameter $N$ is the normal force, $f_d$ is the dimensionless dynamic friction coefficient, and the sign function is defined as
\[\sign\alpha := \begin{cases} 1& \quad \text{ if } \alpha>0,\\ 
\num{-1}& \quad \text{ if } \alpha<0.\end{cases}\]
\Cref{fig:stiction} illustrates the slipping law \cref{eq:slip_law}.
For zero slip velocity ($v=0$),  it is necessary to consider whether this happens on a whole time interval or only instantaneously, i.e. whether $\dot{v}$ is also zero or not. The former case ($v=\dot{v}=0$) defines the stick phase, and from \cref{eq:friction_oscillator} it follows that
\be \label{eq:stick_law} 
F =  w(\bar{t},u) \quad \text{ when } \quad v=0 \quad \text{and} \quad \lvert w \rvert < Nf_s ,
\ee
where $w(\bar{t},u) := \kappa u - F_\omega(\bar{t})$ is  the sum of forces that induce the motion of $M$. The parameter $f_s$  in \cref{eq:stick_law} is the dimensionless static friction coefficient and $f_s>f_d>0$ \cite{rabinowicz1951a}. The idea is that the value of the static friction is exactly the one that counteracts the other forces acting on $M$, so that the mass will keep on sticking. However the static friction \cref{eq:stick_law} can only oppose the motion of $M$  up to the maximum static friction $\pm N f_s$, thus  
  \[ \label{eq:stick_instantaneous}
F =  Nf_s \sign{w} \quad \text{ when } \quad v=0 \quad \text{and} \quad \lvert w \rvert > Nf_s.
\]
In this latter case the friction force is not sufficient to maintain $\dot{v}=0$ and therefore the mass will slip in forward time. \Cref{fig:stiction_y0} illustrates the friction law for $v=0$. In compact form, stiction is written as: 
\[
F(v,w) = \begin{cases}
- N f_d \sign{v} & \quad v \neq 0,\\
 w&   \quad  v=0 \text{ and } \lvert w  \rvert < Nf_s ,\\
  Nf_s \sign{w} & \quad  v=0  \text{ and } \lvert w  \rvert > Nf_s.
\end{cases}\]
The friction law is not defined for  $v=0$ and $\lvert w\rvert = N f_s$, where the external forces equal the maximum static friction during stick. Other modelling choices  may fix a value of $F$ in these points. These choices do  not affect the results of the following analysis, see \cref{sec:solution_definition}.
\begin{figure}[t!]
	\centering
	\begin{subfigure}{.48\textwidth}\caption{ }\label{fig:stiction}
  \centering
\includegraphics[scale=1.1]{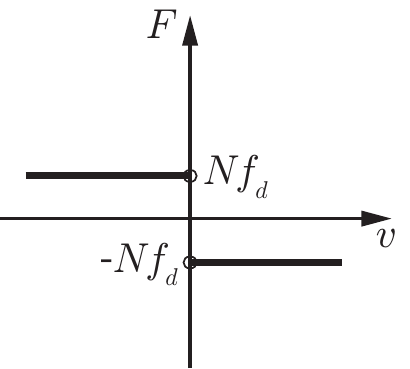}\end{subfigure}
\begin{subfigure}{.48\textwidth}\caption{ }\label{fig:stiction_y0}
  \centering
 \includegraphics{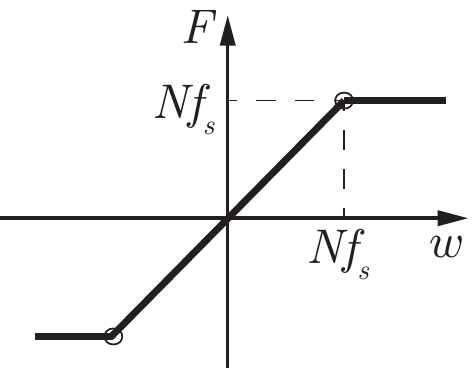}\end{subfigure}
  \caption{Stiction friction $F(v,w)$.  \subref{fig:stiction}: $v\neq 0$.   \subref{fig:stiction_y0}: $v=0$.  }\label{fig:stiction_law}\end{figure}
By  rescaling
\[
u = \frac{V}{\omega} x, \quad v = Vy, \quad \bar{t} =\frac{t}{\omega}, 
\]
system \cref{eq:friction_oscillator} is rewritten in its dimensionless form:
\be \label{eq:model_dimensionless}
\begin{aligned}
x' &= y,\\
y' &= -\xi(x,\theta) + \mu(y,\xi(x,\theta)),\\
\theta' &= 1,
\end{aligned}\ee
where $\theta \in \mathbb{T}^1$ is a new variable describing the phase of the periodic forcing, and that makes system \cref{eq:model_dimensionless} autonomous. Furthermore 
\[
\xi(x,\theta) := \frac{w}{A} = \gamma^2 x + \sin \theta, \]
is the sum of the rescaled external forces, and it is often referred to as $\xi$ in the following analysis.
In this new system the prime has the meaning of differentiation with respect to the time $t$, and $\gamma:= \Omega/\omega$ is the ratio between the natural frequency of the spring $\Omega := \sqrt{\kappa/M}$ and the forcing frequency $\omega$. Therefore $\gamma\to\infty$ corresponds to the rigid body limit. The function $\mu$ describes the dimensionless stiction law:
\be \label{eq:stiction_dimensionless}
\mu(y,\xi) = \begin{cases}
- \mu_d \sign{y} & \quad y \neq 0,\\
 \xi&   \quad  y=0 \text{ and } \lvert \xi  \rvert < \mu_s ,\\
  \mu_s \sign{\xi} & \quad  y=0  \text{ and } \lvert \xi  \rvert > \mu_s,
\end{cases}
\ee
where $\mu_{d,s} := Nf_{d,s} / A$.
System \cref{eq:model_dimensionless} together with the friction function \cref{eq:stiction_dimensionless} is the model used in the  rest of the  analysis. In  compact form it is written as $z' = Z(z)$, where $z := (x,y,\theta) \in \mathbb{R}^2 \times \mathbb{T}^1$, and $\mathbb{T}^1 := \mathbb{R} / 2\pi\mathbb{Z}$. The vector field   $Z(z)$  is not defined on the two lines $\{  y=0, \xi=\pm\mu_s\}$.  \Cref{sec:PWS_analysis} studies the phase space of \cref{eq:model_dimensionless} using geometrical tools from piecewise-smooth theory \cite{dibernardo2008a,filippov1988a}.  


\section{Geometric analysis of the discontinuous system}\label{sec:PWS_analysis}
This section analyses  the friction oscillator \cref{eq:model_dimensionless} with stiction friction \cref{eq:stiction_dimensionless} in the context of piecewise-smooth dynamical systems. The notation is consistent with the  one in \cite{guardia2010a}.
System \cref{eq:model_dimensionless}  is smooth in the two regions
\[
\begin{aligned}
G^+ &:= \{ (x,y,\theta) \in \mathbb{R}^2 \times \mathbb{T}^1 \mid y>0\},\\
G^- &:= \{ (x,y,\theta) \in \mathbb{R}^2 \times \mathbb{T}^1 \mid y<0\}.\\
\end{aligned}
\]
Let $Z^+(z)$ ($Z^-(z)$) be the vector field $Z(z)$   restricted to ${G}^+ $ (${G}^- $) and extended to the closure of ${G}^+ $ (${G}^- $). These two smooth vector fields have the explicit form
\[ Z^\pm =\begin{cases}
x' &= y,\\
y' &= -\xi(x,\theta) \mp \mu_d,\\
\theta' &= 1.\end{cases}\]
The set $\Sigma := \{ (x,y,\theta) \in \mathbb{R}^2 \times \mathbb{T}^1 \mid y=0\}
$ is a surface of discontinuity of $Z(z)$  and it is called the {\it switching manifold}. The vector field $Z(z)$ is well-defined  in $\Sigma \setminus \{ \xi = \pm\mu_s\}$ and its dynamics on the $y$-coordinate is
\[ y' = -\xi(x,\theta)  + \mu\left(0,\xi(x,\theta)\right) \,\, \begin{cases} 
>0 &\quad \text{for} \quad \xi<-\mu_s,\\
=0 & \quad \text{for} \quad \lvert\xi\rvert<\mu_s,\\
<0&\quad \text{for} \quad \xi>\mu_s.\end{cases}\]
Therefore it is natural to subdivide $\Sigma$ into the   three sets  
\[\begin{aligned}
\Sigma_c^+ &:= \{ (x,y,\theta) \in \mathbb{R}^2 \times \mathbb{T}^1 \mid y=0 \, \text{ and } \, \xi < -\mu_s \},\\
\Sigma_s &:= \{ (x,y,\theta) \in \mathbb{R}^2 \times \mathbb{T}^1 \mid y=0 \, \text{ and } \,  -\mu_s < \xi < \mu_s \},\\
\Sigma_c^- &:= \{ (x,y,\theta) \in \mathbb{R}^2 \times \mathbb{T}^1 \mid y=0 \, \text{ and } \, \xi > \mu_s\},
\end{aligned}\]
that are shown in \cref{fig:phasespace_tangencies}. The set $\Sigma_c^+$ ($\Sigma_c^-$) is called   the {\it crossing region  pointing \it upwards} ({\it  downwards}), because orbits here switch from   $G^-$   to $G^+$ (from $G^+$ to $G^-$). The strip  $\Sigma_s$  is called the {\it sticking region} because trajectories within it are not allowed to switch to $G^\pm$, and they correspond to solutions where the mass sticks to the table.
\begin{figure}[t!]
\centering
\begin{subfigure}{0.9\textwidth}\caption{ }\label{fig:phasespace_tangencies}  \centering
\includegraphics[width=\textwidth]{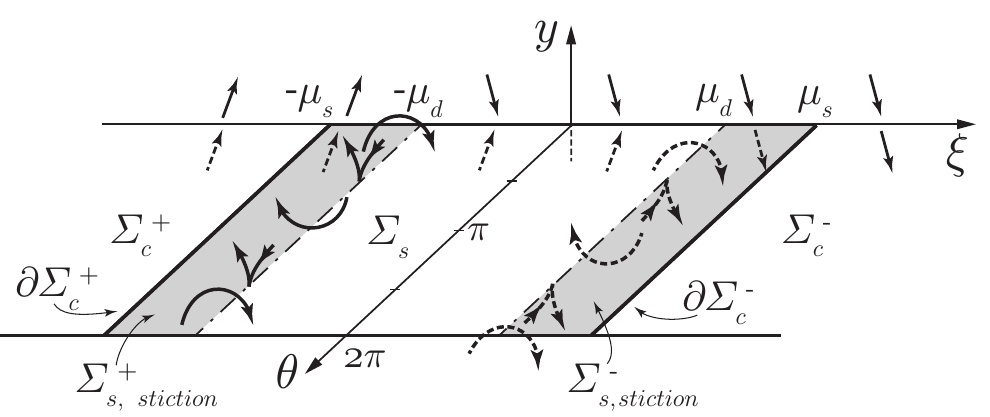}\end{subfigure}
\begin{subfigure}{.9\textwidth}\caption{ }\label{fig:sticking_vectorfield}\centering
\includegraphics[width=0.8\textwidth]{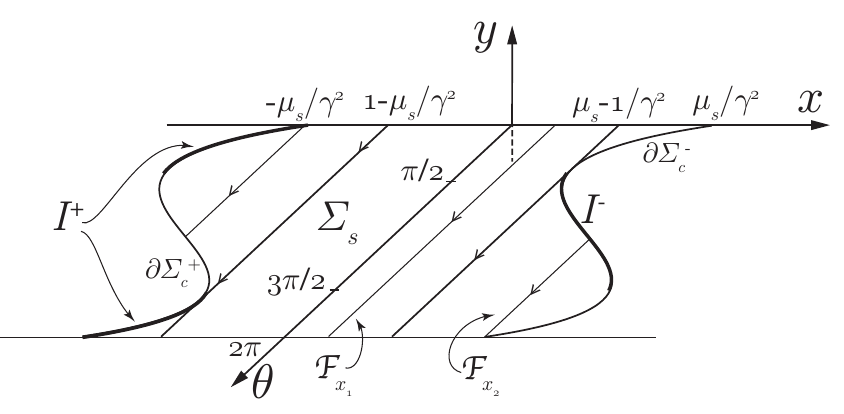}\end{subfigure}
\caption{
 \subref{fig:phasespace_tangencies}: Vector fields $Z^\pm$ and their tangencies at $\xi = \mp \mu_d$ in the $(\xi,y,\theta)$-space.  $Z^-$ is dashed because it is below $\Sigma_s$. The  grey bands indicate   where $Z^\pm$ suggest crossing but instead the solution for $y=0$ is sticking. \subref{fig:sticking_vectorfield}: Phase space of $Z_s$ in the $(x,y,\theta)$-space with the tangencies at $\theta = \{\pi/2, 3\pi/2\}$.  The leaf $\mathcal{F}_{x_1}$ is a full circle, while $\mathcal{F}_{x_2}$ is an arc of a circle. The intervals of non-uniqueness $I^\pm$ are introduced in \cref{prop:existence_uniqueness}.}\label{fig:phasespace_PWS}
\end{figure}  
Let $Z_s(z)$ be the smooth vector field  $Z(z)$  restricted to ${\Sigma}_s$ and extended to the closure of ${\Sigma}_s$. This  two-dimensional vector field  has the explicit form $(x,\theta)'=(0,1)$,  thus $\Sigma_s$ is foliated by invariant arcs of circles
\be\label{eq:foliation} \mathcal{F}_{x_0} := \{ (x,y,\theta) \in \Sigma_s \mid x = x_0 \}, \ee
since $\theta\in \mathbb{T}^1$. \Cref{fig:sticking_vectorfield} shows the foliation $\mathcal{F}_{x_0}$.  The boundaries of $\Sigma_s$ with $\Sigma_c^\pm$ define the two sets  
\[
\begin{aligned}
\partial\Sigma_c^+ &:= \{ (x,y,\theta) \in \mathbb{R}^2 \times \mathbb{T}^1 \mid y=0 \, \text{ and } \, \xi = -\mu_s \},\\
\partial\Sigma_c^- &:= \{ (x,y,\theta) \in \mathbb{R}^2 \times \mathbb{T}^1 \mid y=0 \, \text{ and } \, \xi = \mu_s\}.
\end{aligned}
\]
The vector field  $Z(z)$ is not defined on  $\partial\Sigma_c^\pm$, but the  three vector fields   $Z_s(z)$ and $Z^\pm(z)$ are.  Indeed   $\partial\Sigma_c^\pm$ belong to the closure of both $\Sigma_s$ and $G^\pm$.   
Hence on $\partial\Sigma_c^\pm$ solutions may be forward non-unique. This will be discussed in \cref{sec:solution_definition}. \\
The following two \cref{prop:tangency_sigmas,prop:tangency_sigmacpm} say where the vector fields $Z_s(z)$, $Z^\pm(z)$ are tangent to $\partial \Sigma_c^\pm$ and $\Sigma$ respectively. The results are shown in \cref{fig:phasespace_PWS}. First, a definition introduces the concepts of visible and invisible tangency.
\begin{definition}
Let   $\hat{\Sigma}:= \{z \in \mathbb{R}^n \mid \chi(z) >0 \}$,  where  $\chi : \mathbb{R}^n\to \mathbb{R}$ 
is a smooth and regular function such that $\nabla \chi(z) \neq 0$ for every $z\in \mathbb{R}^n$.
Furthermore   let $\hat{Z} : \hat{\Sigma}\to\mathbb{R}^n$ be a smooth  vector field, having a smooth extension to the boundary of $\hat{\Sigma}$, that is for $\chi(z)=0$.  In addition, let $\mathcal{L}_{\hat{Z}}\chi(z):= \nabla \chi \cdot \hat{Z}(z)$  denote the Lie derivative    of $\chi$ with respect to  $\hat{Z}(z)$.\\
The vector field $\hat{Z}(z)$ is  \textnormal{tangent} to the set $\chi(z)=0$ at $p\in\hat{\Sigma}$ if  $\mathcal{L}_{\hat{Z}}\chi(p)=0$. The tangency is called \textnormal{visible} (\textnormal{invisible}) if  $\mathcal{L}_{\hat{Z}}^2\chi(p)>0$ ($\mathcal{L}_{\hat{Z}}^2\chi(p)<0$), where $\mathcal{L}_{\hat{Z}}^2\chi(p)$ is the second order Lie derivative.  The tangency is a \textnormal{cusp} if $\mathcal{L}_{\hat{Z}}^2\chi(p)=0$ but $\mathcal{L}_{\hat{Z}}^3\chi(p)\neq0$.  
\end{definition}
 In other words, the tangency is visible if the orbit $z' = \hat{Z}(z)$ starting at $p$ stays in $\hat{\Sigma}$ for all sufficiently small $\lvert t \rvert >0$, and it is invisible if it never does so  \cite[p.~93 and  p.~237]{dibernardo2008a}.  A quadratic tangency is also called a {\it fold}  \cite{teixeira1993a}.  
 \begin{proposition}\label{prop:tangency_sigmas}
$Z_s(z)$ is tangent to $\partial \Sigma_c^-$ ($\partial \Sigma_c^+$) in the isolated points  $\theta \in\{\pi/2, 3\pi/2\}$.  The tangency is visible (invisible) for $\theta =\pi/2$, and invisible (visible) for $\theta = 3\pi/2$. 
\end{proposition}
\begin{proof}
Define the function $\chi(\xi,\theta) = \mu_s-\xi(x,\theta)$ so that it is defined within  $\Sigma$, and its zeroes belong to $\partial\Sigma_c^-$. Then $\mathcal{L}_{Z_s}\chi(p)=0$ in $\theta = \{\pi/2, 3\pi/2\}$. Moreover $\mathcal{L}^2_{Z_s}\chi(p) = \sin\theta $. Hence $\theta = \pi/2$ ($\theta = 3\pi/2$) is a visible (invisible) fold.   Similar computations prove the result for $\partial \Sigma_c^+$. 
\end{proof}
\begin{corollary}\label{cor:sticking_pos}
If $\mu_s>1$, then  the invariant leaves $\mathcal{F}_x$ of \cref{eq:foliation} with $\lvert \gamma^2 x  \rvert < \mu_s-1 $ are  periodic  with period $2\pi$. The remaining leaves   of \cref{eq:foliation}, having $\lvert \gamma^2 x  \rvert \geq \mu_s-1 $, escape   $\Sigma_s$ in finite time. If $\mu_s <1$ no periodic solutions exist on $\Sigma_s$. 
\end{corollary}
\begin{proof} The sticking trajectory $\gamma^2 x(t) =\mu_s-1$ ($\gamma^2 x(t) =-\mu_s+1$) is tangent to $\partial \Sigma_c^-$ ($\partial \Sigma_c^+$) because $\xi(x,\pi/2)=\mu_s$ ($\xi(x,3\pi/2)=-\mu_s$).   These two lines coincide for $\mu_s=1$. When $\mu_s>1$ the orbits $\lvert \gamma^2 x(t) \rvert < \mu_s-1 $ are included within the two tangent orbits. Hence they never intersect the boundaries $\partial \Sigma_c^\pm$  and therefore are periodic with period $2\pi$. Instead, the trajectories $\mu_s>\lvert \gamma^2 x(t) \rvert \geq\mu_s-1$ exit $\Sigma_s$ in finite time.  \end{proof}
The orbit $\mathcal{F}_{x_1}\subset \Sigma_s$ of \cref{fig:sticking_vectorfield} is periodic, while $\mathcal{F}_{x_2}$ leaves $\Sigma_s$ in finite time. The  period ${ T} = 2\pi$ corresponds  to a period $\bar{T} = 2\pi/\omega$ in the original time $\bar{t}$,  as it is often mentioned in the literature \cite{csernak2006a ,shaw1986a}.   The condition $\mu_s>1$     corresponds to $N f_s > A$ that is, the maximum static friction force is larger than the amplitude of the  forcing $F_\omega$. This interpretation makes it an obvious condition for having sticking solutions.
 \begin{proposition}\label{prop:tangency_sigmacpm}
The vector field $Z^-$  ($Z^+$)  is tangent to $\Sigma$ on  the line $\xi = \mu_d$ ($\xi =- \mu_d$). The tangency is   invisible  (visible)   for   $\theta \in ]\pi/2, 3\pi/2[$, it is visible (invisible) for      $\theta\in[0,\pi/2[$ and $\theta\in]3\pi/2,2\pi[$, while it is a cusp on the isolated points  $\theta = \{\pi/2, 3\pi/2\}$.
\end{proposition}
\begin{proof}
Define the function $\chi(x,y,\theta)=-y$  so that it is defined in $G^-$ and it is zero in $\Sigma$. Then $\mathcal{L}_{Z^-}\chi(p)  = \xi(x,\theta) -  \mu_d = 0$   on the line $\xi = \mu_d, \,\theta\in\mathbb{T}^1$.   Moreover $\mathcal{L}_{Z^-}^2\chi(p)  = \cos \theta$.  This is negative for $\theta \in ]\pi/2, 3\pi/2[$ and positive for $\theta\in[0,\pi/2[$ and $\theta\in]3\pi/2,2\pi[$. The points $\theta = \pi/2$ and $\theta= 3\pi/2$ have $\mathcal{L}_{Z^-}^2\chi(p) =0$ but $ \mathcal{L}_{Z^-}^3\sigma(p)\neq0$.  Similar computations prove the result  for  $Z^+(z)$. 
\end{proof}
The knowledge of the tangencies is sufficient to describe the local phase space of system \cref{eq:model_dimensionless} around the discontinuity $\Sigma$, as \cref{fig:phasespace_PWS} shows. 
\Cref{sec:solution_definition} discusses how forward solutions of $Z(z)$, that are smooth within each set $G^\pm$ and $\Sigma_s$, connect at the boundaries of these regions.  It is futile to study solutions in backwards time, because when an orbit lands on $\Sigma_s$, the information of when it has landed is lost.
\section{Forward solutions of the discontinuous system}\label{sec:solution_definition}
Classical results on existence and uniqueness of solutions require Lipschitz continuous right hand sides, and therefore do not apply to discontinuous systems like  \cref{eq:model_dimensionless}. 
A class of discontinuous systems for which some results 
are known,  is the one of Filippov-type \cite{filippov1988a}.  In a Filippov-type system the vector fields $Z^\pm(z)$ are sufficient to describe the dynamics within the switching manifold $\Sigma$. This is useful especially when there is no vector field already defined on $\Sigma$.  Let $Z_y^\pm(z)$ be the $y$ component of $Z^\pm(z)$ in a point $z\in\Sigma$. Then  Filippov's convex method defines the {\it crossing region} as the subset of $\Sigma$ where $Z_y^+\cdot Z_y^-(z)>0$, while the {\it sliding region} $\Sigma_{s,\textnormal{Filippov}}$ satisfies $Z_y^+\cdot Z_y^-(z)<0$ \cite[\S~2]{filippov1988a}, \cite[p.~76]{dibernardo2008a}. The idea is that solutions inside the sliding region  cannot exit $\Sigma$ because  $Z^\pm(z)$ do not allow it.
\begin{remark}\label{rem:non_Filippov}
System  \cref{eq:model_dimensionless} together with the friction law \cref{eq:stiction_dimensionless} is not of Filippov-type. 
Indeed the  {sliding} region of system \cref{eq:model_dimensionless}  is
\[ \Sigma_{s,\textnormal{Filippov}} := \{ (x,y,\theta) \in \mathbb{R}^2 \times \mathbb{T}^1 \mid y=0 \, \text{ and } \, -\mu_d < \xi < \mu_d \} ,\]
 that is a strip within $\Sigma_s$ whenever $\mu_d<\mu_s$. In the two remaining bands
 \[ \begin{aligned}
  \Sigma_{s,\textnormal{stiction}}^- &:= \{ (x,y,\theta) \in \mathbb{R}^2 \times \mathbb{T}^1 \mid y=0 \, \text{ and } \, \xi \in ]\mu_d,\mu_s[ \}  ,\\
  \Sigma_{s,\textnormal{stiction}}^+ &:= \{ (x,y,\theta) \in \mathbb{R}^2 \times \mathbb{T}^1 \mid y=0 \, \text{ and } \, \xi \in ]-\mu_s,-\mu_d[ \} ,
  \end{aligned}\]   that are coloured in grey in \cref{fig:phasespace_tangencies}, the vector field $Z_s(z)$  does not belong to the convex closure of  $Z^\pm(z)$. Here Filippov's method predicts orbits to switch   from $G^+$ to $G^-$ or vice versa, but the actual solution of model \cref{eq:model_dimensionless}  lies within $\Sigma_s$. 
 When $\mu_d= \mu_s$ the friction law \cref{eq:stiction_dimensionless}  equals the classical Coulomb friction and $\Sigma_s$ coincides with $\Sigma_{s,\textnormal{Filippov}}$. This case has been studied in \cite{guardia2010a,kowalczyk2008a,csernak2007a}. 
 \end{remark}
The two grey bands $ \Sigma_{s,\textnormal{stiction}}^\pm$  are unstable to perturbations in  $y$. Consider for instance a trajectory in $\Sigma_{s,\textnormal{stiction}}^-$ that is pushed to $G^-$ by an arbitrary small perturbation: this solution will evolve far   from $ \Sigma_{s,\textnormal{stiction}}^-$ by following $Z^-(z)$. \\
Another notion of forward solution of a discontinuous system  is the  {\it Carath\'eodory solution}  \cite{cortes2008a}, \cite[\S 1]{filippov1988a}.  This is an absolutely continuous function $z(t)$ that satisfies 
\be \label{eq:caratheodory} 
z(t) = z(0) + \int_{0 }^t Z(z(s))\,ds, \quad t\geq 0,\ee
where the integral is in a Lesbegue sense. Hence in order to have a Carath\'eodory solution, $Z(z)$ needs only to be defined almost everywhere. 
\begin{proposition}
For every $z_0\!=\! z(0)  \in \mathbb{R}^2\times\mathbb{T}^1$ there exists a global forward Carath\'eodory solution of model \cref{eq:model_dimensionless} satisfying \cref{eq:caratheodory} for every $t\geq0$. 
\end{proposition}
\begin{proof}
For every $z_0$ there exists at least one local classical solution of either $Z^\pm(z)$ or $Z_s(z)$. A forward solution of \cref{eq:caratheodory} is obtained  by piecing together such local orbits together on $\Sigma$. This can be done for every $t>0$ since $Z^\pm(z)$ and $Z_s(z)$ are each linear in $(x,y)$, excluding the possibility of blowup in finite time. 
\end{proof}
  \begin{figure}[t!]
\centering
\begin{subfigure}{.48\textwidth}\caption{ }\label{fig:artifacts_solutions}
\includegraphics[width=\textwidth]{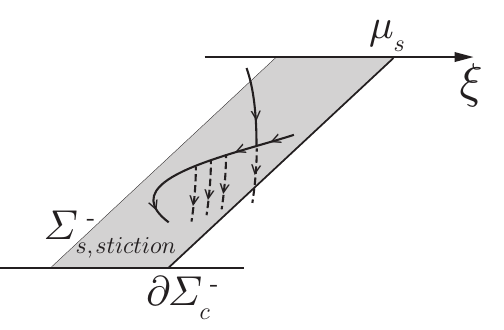}\end{subfigure}
\begin{subfigure}{.48\textwidth}\caption{ }\label{fig:non_uniqueness}
\includegraphics[width=0.9\textwidth]{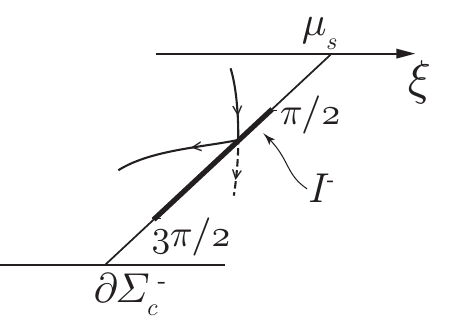}\end{subfigure}
\caption{\subref{fig:artifacts_solutions}: A  Carath\'eodory  solution with a pathological non-determinacy of the forward motion on the grey band. \subref{fig:non_uniqueness}: Stiction solutions interacting with the line of forward non-uniqueness $I^-$.    }\label{fig:non_unique}
 \end{figure}
Not every forward Carath\'eodory solution has a physical   meaning.  Consider for instance  a trajectory   that under the forward flow \cref{eq:model_dimensionless} lands inside $\Sigma_{s,\textnormal{stiction}}^- $, as shown in \cref{fig:artifacts_solutions}. There are two ways to obtain a forward solution  at this point: either leave $\Sigma$ and follow the vector field $Z^-(z)$, or remain on $\Sigma_s$.    Besides, the forward trajectory on $\Sigma_s$ may switch to $G^-$ at any point  within $\Sigma_{s,\textnormal{stiction}}^-$.  The orbits switching to $G^-$ appear to be mathematical artifacts, as they do not satisfy the  condition $\lvert \xi\rvert > \mu_s$ of the stiction law \cref{eq:stiction_dimensionless}. 
 There is a need to have a concept of solution that discards all these pathologies. The following definition does so, by using a ``minimal'' approach.  
\begin{definition} 
A  \textnormal{stiction solution}  $t \mapsto z(t)$, with $t\geq0$, is a  Carath\'eodory solution that  leaves $\Sigma_s$ only at the boundaries $\partial\Sigma_c^\pm$. \\
A stiction solution is called {\it singular} if for some $t_1\geq0$ the point   $z(t_1)$ belongs to one of the following sets
\[\begin{aligned}
I^+ &:= \{ (x,y,\theta) \in \mathbb{R}^2\times \mathbb{T}^1 \mid \xi = -\mu_s, y = 0, \theta \in[\pi/2,3\pi/2]\,\, \},\\ 
I^- &:= \{ (x,y,\theta) \in \mathbb{R}^2\times \mathbb{T}^1 \mid \xi = \mu_s, \,\,\,\,\, y = 0, \theta \in[0,\pi/2] \cup [3\pi/2,2\pi[\,\,\}.
\end{aligned} \]
Otherwise, the stiction solution is called {\it regular}.
\end{definition}
The sets $I^\pm$ belong to the boundary lines $\partial \Sigma_c^\pm$. Three vector fields  are defined on $\partial \Sigma_c^\pm$: $Z_s(z)$ and $Z^\pm(z)$. In particular on both $I^\pm$ the vector field $Z_s(z)$ points inside  $\Sigma_s$, as it follows from   \cref{prop:tangency_sigmas}, compare with \cref{fig:sticking_vectorfield}. \Cref{prop:existence_uniqueness} describes the existence and uniqueness of stiction solutions for model \cref{eq:model_dimensionless}. 

\begin{proposition}\label{prop:existence_uniqueness}
There exists a stiction solution $z(t)$ of problem \cref{eq:model_dimensionless}  for any initial initial condition $z_0 =z(0) \in \mathbb{R}^2\times\mathbb{T}^1$. Regular stiction solutions are forward unique, while singular stiction solutions are forward non-unique. 
\end{proposition}
\begin{proof}
Stiction solutions  are Carath\'eodory solutions, hence they exist. Consider a trajectory $z(t)$ that reaches $ I^-$ at a time $t_1$, as shown in \cref{fig:non_uniqueness}. Two different forward solutions   satisfy \cref{eq:caratheodory}: either leave $\Sigma$ and follow the vector field $Z^-(z)$, or remain on $\Sigma_s$.  Hence the singular stiction solution is forward non-unique. Similarly for $I^+$. On the contrary, if $z(t) \notin I^\pm$ at any $t\geq0$, then  there is always only one way to piece together the vector fields at the boundaries $\partial\Sigma_c^\pm$ and therefore $z(t)$ is forward unique.
 \end{proof}
The  non-uniqueness of models with stiction friction has  been  mentioned  in \cite{bliman1995a,olsson1997a}, without any further explanation.  It is not possible to predict whether,  for singular stiction solutions,  the mass will slip or stick in forward time. Hence numerical simulations that use stiction friction have to make a choice at the points of non-uniqueness to compute the forward flow, often without noticing that a choice is made. This means that   solutions may unawarely be discarded. 
\Cref{sec:regularization} investigates the non-uniqueness by regularization.    

\section{Regularization}\label{sec:regularization}
A regularization of the vector field $Z(z)$ is a $1$-parameter family $Z_\vare(z)$ of smooth vector fields  defined by 
\be \label{eq:regularization_combination} 
Z_\vare(z) := \frac{1}{2} Z^+(z)(1+\phi(\vare^{-1}y)) + \frac{1}{2} Z^-(z)(1-\phi(\vare^{-1}y)), \ee 
for $0<\vare\ll1$. The function $\phi(y)$ is an odd, $C^k$-function $(1\leq k \leq \infty)$ that satisfies
\be\label{eq:regularization} \phi(y) = \begin{cases}
1,&  y \geq1,\\
\mu_s/\mu_d,&  y = \delta,
\end{cases} \quad \text{and} \quad
\phi'(y) \,\, \begin{cases}
>0,&   0<y <\delta,\\
=0,&   y= \delta,\\ 
<0,&  \delta<y <1,
\end{cases} \quad  \phi''(\delta)<0, \ee
where $0<\delta<1$. This function is shown in  \cref{fig:regularization}.   The regularized problem $z'=Z_\vare(z)$ has the advantages of being smooth, and of approximating the discontinuous problem \cref{eq:model_dimensionless}  for $0<\vare\ll1$. In particular,  by the first property of \cref{eq:regularization}, it follows that   $Z_\vare(z) = Z^\pm(z)$ for $y\gtrless \pm\vare$, so that the two problems coincide outside of the {\it region of regularization} $y \in ]-\vare,\vare[$. In non-compact form $z'=Z_\vare(z)$   is the singularly perturbed problem
\be \label{eq:vectorfield_regularized} 
\begin{aligned}
x' &= y,\\
y' &= - \xi(x,\theta) - \mu_d \phi(\vare^{-1}y),\\
\theta' &= 1,
\end{aligned}\ee
with $\xi(x,\theta)= \gamma^2 x + \sin\theta$ the function introduced in \cref{sec:model}. 
 \begin{figure}[t!]
 \centering
\includegraphics[width=0.3\textwidth]{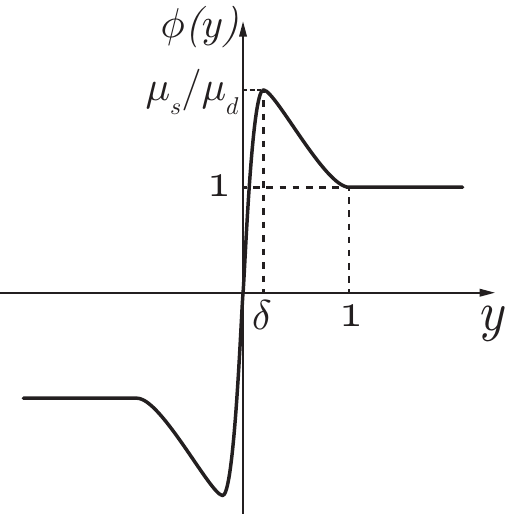}
\caption{A regularization function $\phi(y)$.}\label{fig:regularization}
 \end{figure}
When solutions of \cref{eq:vectorfield_regularized} enter the region of regularization, it is easier to follow them in the rescaled coordinate $\hat{y} = \vare^{-1}y$ so that  $y=\pm\vare$ become $\hat{y}=\pm1$.  
In the new scale, system \cref{eq:vectorfield_regularized} becomes the multiple time scales problem 
\be\label{eq:slow_problem}
\begin{aligned}
x' &=  \vare \hat{ y},\\
\vare \hat{y}' &= - \xi(x,\theta) - \mu_d \phi(\hat{y}),\\
\theta' &= 1,
\end{aligned}
\ee
 that  is also known as the {\it slow problem} \cite{kuehn2015a,jones1995a}.  By introducing the  fast time $\tau:= t/\vare$,  system \cref{eq:slow_problem} is equivalent to the {\it fast problem} 
\be\label{eq:fast_problem}
\begin{aligned}
\dot{x} &=   \vare^2 \hat{y} ,\\
\dot{ \hat{y}} &= - \xi(x,\theta) - \mu_d \phi(\hat{y}),\\
\dot{\theta} &= \vare,
\end{aligned}
\ee
with the overdot meaning the differentiation with respect to the fast time $\tau$.
The parameter $\vare$  measures both the perturbation from the discontinuous system, as in equation \cref{eq:regularization_combination}, and the separation of the time scales. The standard procedure  for solving  multiple time scales problems is to combine the solutions of the {\it layer problem} 
 \be\label{eq:layer_problem}
\dot{ \hat{y}} = - \xi(x,\theta) - \mu_d \phi(\hat{y}), \quad (x,\theta)(\tau_0) = (x_0,\theta_0),
\ee
with the ones of the {\it reduced problem}
\be\label{eq:reduced_problem}
\begin{aligned}
x' &=  0,\\
0 &= - \xi(x,\theta) - \mu_d \phi(\hat{y}),\\
\theta' &= 1,
\end{aligned}
\ee
where \cref{eq:layer_problem} and \cref{eq:reduced_problem} are the limit for $\vare\to0$ of  the fast and slow problems \cref{eq:fast_problem} and \cref{eq:slow_problem}.
\begin{figure}[t!]
 \centering
\includegraphics[width=0.55\textwidth]{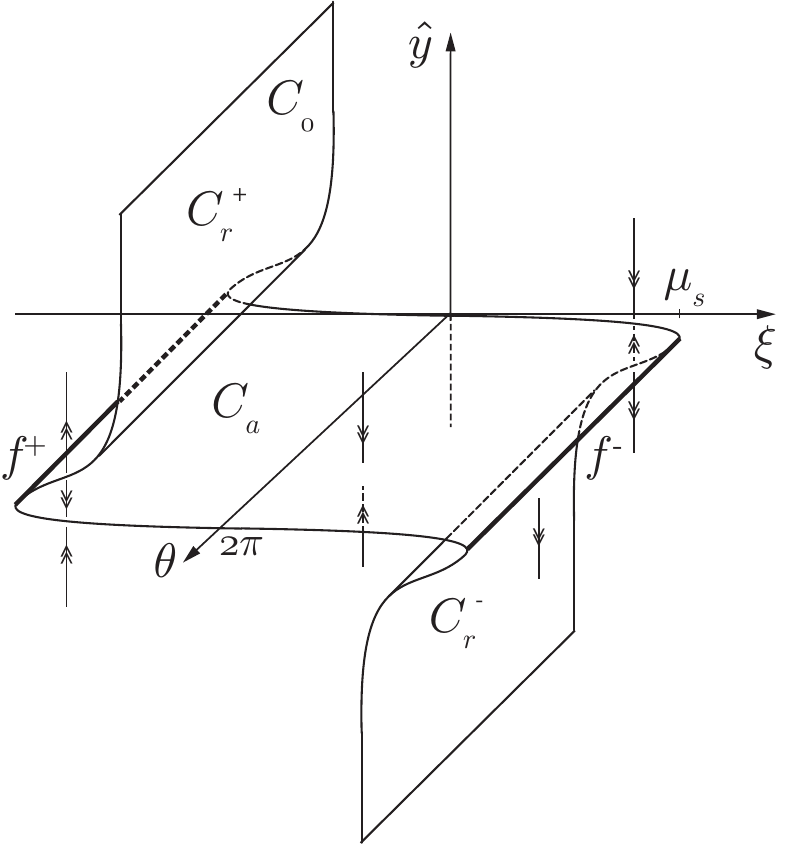}
\caption{Critical manifold $C_0$ and its stability properties. In bold: $f^\pm$. The double arrow denotes dynamics in the fast time $\tau$.}\label{fig:critical_manifold}
 \end{figure}
  The set of fixed points of the layer problem \cref{eq:layer_problem} is called  the {\it critical manifold}
\be\label{eq:critical_manifold}
C_0 := \{ (x,\hat{y},\theta)\in\mathbb{R}^2\times\mathbb{T}^1 \mid \quad \xi(x,\theta) + \mu_d \phi(\hat{y}) =0 \}, 
\ee
and the solutions of the reduced problem \cref{eq:reduced_problem} are constrained to it. The critical manifold is said to be {\it normally hyperbolic} in the points where
\[ \frac{\partial \dot{\hat{y}}}{\partial \hat{y}} \bigg\rvert_{C_0} = -\mu_d  \phi'(\hat{y}^{C_0})\]
is non zero, and $ \hat{y}^{C_0} = \phi^{-1}(-\xi(x,\theta)/\mu_d)$.  It follows that $C_0$ is not normally hyperbolic on the two {\it fold lines} 
\[ f^\pm := \{(x,\hat{y},\theta) \in \mathbb{R}^2\times\mathbb{T}^1 \mid \xi = \mp \mu_s, \hat{y} = \pm \delta \}. \]
These lines separate $C_0$ into the three invariant sets of \cref{eq:layer_problem}
\[ \begin{aligned}
C_r^+ &:= \{ (x,\hat{y},\theta) \in C_0 \mid  \delta<\hat{y}<1\},\\
C_a &:= \{ (x,\hat{y},\theta) \in C_0 \mid -\delta<\hat{y} <\delta\},\\
C_r^- &:= \{ (x,\hat{y},\theta) \in C_0 \mid  -1<\hat{y}<-\delta\},\\
\end{aligned}\]
as shown in \cref{fig:critical_manifold}, where $C_a$ is attracting and  $C_r^{\pm}$ are repelling. Notice that  $C_a$ is a graph  $\hat{y} \in \,\,]-\delta,\delta[$ over  $\Sigma_s$, while   $C_r^+$ ($C_r^-$) is a   graph  $ \hat{y}  \in\,\,]\delta,1[$ ($ \hat{y}  \in\,\,]-1,-\delta[$) over  $\Sigma_{s,\text{stiction}}^+$ ($\Sigma_{s,\text{stiction}}^-$). In terms of $(x, y, \theta)$, these sets  collapse onto $\Sigma_s$ and  $\Sigma_{s,\text{stiction}}^\pm$ respectively as $\vare\to0$, since $y=\vare\hat{y}$. Similarly, $f^\pm$ collapse onto $\partial \Sigma_c^\pm$.  This means that  in the $(x,y,\theta)$-space it is not possible to distinguish whether a trajectory belongs to $C_a$ or to  $C_r^\pm$ for $\vare=0$. 
\begin{proposition} \label{prop:collapse_epsilon} The reduced problem on $C_0$ coincides with the  vector field $Z_s(z)$ on $\Sigma_s$.  \end{proposition}
The proof is straightforward since  the reduced  problem, once  constrained to $C_0$, is  $(x', \theta') = (0, 1)$. 
From this Proposition, and the fact that $Z_\vare(z) = Z^\pm(z)$ for $y\gtrless \pm\vare$, it follows that the regularized problem \cref{eq:vectorfield_regularized} captures all the main  features of the discontinuous vector field \cref{eq:model_dimensionless} for $\vare\to0$. Furthermore, when $0<\vare\ll1$ the solutions of \cref{eq:vectorfield_regularized} are uniquely defined, so that the issue of   non-uniqueness   of \cref{eq:model_dimensionless} is eliminated.   \Cref{prop:collapse_epsilon} also motivates the  conditions \cref{eq:regularization} for the function $\phi(y)$, as explained in the following Remark. 
 \begin{remark}\label{rem:sotomayor_texeira}
The well known Sotomayor and Teixeira  (ST) regularization, considers a regularization function  $\phi^{ST}(y)$  that is monotonously increasing in $y\in]-1,1[$  \cite{sotomayor1996a}. At the singular limit, the regularization $Z_\vare^{ST}(z)$ has an  attracting invariant manifold $C_a^{ST}$ that is   a graph of $\hat{y} $ over  $\Sigma_{s,\text{Filippov}}$ \cite{llibre2008a,kristiansen2014a}. In terms of $(x,y,\theta)$   this set collapses onto $\Sigma_{s,\text{Filippov}}$ instead of $\Sigma_s$, and hence  $Z_\vare^{ST}(z)$ does not tend to $Z(z)$ as $\vare\to0$. For this reason the ST regularization is inadequate for model \cref{eq:model_dimensionless}.
\end{remark}
The results of Fenichel \cite{fenichel1974a,fenichel1979a} guarantee that for $\vare=0$, a normally hyperbolic, compact and invariant manifold $S_0 \subset C_0$ perturbs into a non-unique and  invariant {\it slow manifold} $S_\vare$, that is $\vare$-close to $S_0$ for $\vare$ sufficiently small. Furthermore, system \cref{eq:fast_problem} has an invariant foliation with base on $S_\vare$, that is a perturbation of the foliation of the layer problem \cref{eq:layer_problem} with base on $S_0$.\\ 
Let $\varphi_t(z_0)$ be a regular stiction solution of model \cref{eq:model_dimensionless} with initial condition in $z_0$, and let $\varphi_{t}^{\vare}(z_0)$ be the solution of the regularized problem \cref{eq:vectorfield_regularized} for the same initial condition. The following statement relates these two solutions. 
 \begin{proposition}\label{prop:regular_stiction_solution}
For any $T>0$ there exists an $\vare_0>0$ such that the distance between the two solutions $\varphi_{t}^{\vare}(z_0)$ and $\varphi_t(z_0)$ is bounded by: $\lvert \varphi_t^\vare(z_0)-\varphi_t(z_0)\rvert \leq c(T) \vare^{2/3}$ for $t\in[0,T]$, where $ c(T)$ is a constant that depends upon $T$, and  $0<\vare\leq\vare_0$.
\end{proposition}
\begin{proof}
Fenichel's theorems guarantee that, sufficiently far from the fold lines $f^\pm$, the orbit $\varphi_t^\vare(z_0)$ of the slow-fast problem \cref{eq:slow_problem} is $\mathcal{O}(\vare)$-close to the singular trajectory $\varphi_t(z_0)$. At the folds $f^\pm$, if at the singular level the solutions are unique, the result by  Szmolyan and Wechselberger \cite[Theorem 1]{szmolyan2004a} guarantees  that the distance between the two trajectories is bounded by $\mathcal{O}(\vare^{2/3})$ for a finite time interval $T$. This is the case of regular stiction solutions.   \end{proof}
The following Proposition relates the family of sticking solutions of \cref{cor:sticking_pos} with a family of trajectories on the slow manifold for the regularized problem. For this, define $S_a  \subset C_a$ as the compact, invariant, normally hyperbolic set $ S_a:=\{(x,\hat{y},\theta)\in\mathbb{R}^2\times\mathbb{T}^1\mid \lvert \gamma^2 x \rvert  \leq \mu_s-1-c, \xi(x,\theta) + \mu_d \phi(\hat{y})=0 \}$ for $\mu_s>1$ and $c\in\mathbb{R}^+$  small.  The set $S_a$ is a graph over the set  of  invariant circles of \cref{cor:sticking_pos} for $c\to0$. 
\begin{proposition}\label{prop:sticking_limit_cycle} 
For $0<\vare\ll1$ the  set $S_a$ perturbs into a slow manifold  $S_{a,\vare}$ and on it,  there exists a unique, attracting $2\pi $-periodic limit cycle passing through $(x,\theta)=(0,0) + \mathcal{O}(\vare)$. \end{proposition}
\begin{proof}
From \cref{prop:collapse_epsilon} and \cref{cor:sticking_pos} it follows that $S_a$ is filled by circular trajectories. By Fenichel's results, when $0<\vare\ll1$ the set $S_a$ perturbs  into the graph $\hat{y} =  \phi^{-1}(-\xi(x,\theta)/\mu_d) + \vare h_1(x,\theta)$. On this graph the slow problem \cref{eq:slow_problem} is a $2\pi$-periodic, non-autonomous ODE for $x(\theta)$, where $\theta$ has the meaning of time:
\be\label{eq:xdynamicsepsilon} x'(\theta) = \vare\phi^{-1}\left(\frac{-\xi(x,\theta)}{\mu_d}\right) + \vare^2 h_1(x,\theta).\ee
Fix a global Poincar\'e section at $\theta=0$, and define the return map $P(x(0),\vare) = x(2\pi)$. The fixed points of this map for $0<\vare\ll1$ are the zeros of the function 
\[Q(x(0),\vare):= \frac{P(x(0),\vare)  - x(0)}{\vare}  = \int_0^{2\pi}  \phi^{-1}\left(\frac{-\gamma^2 x(s) - \sin s }{\mu_d}\right)\,ds + \mathcal{O}(\vare), \]
where   the last equality is obtained by integrating \cref{eq:xdynamicsepsilon}.
For $\vare=0$,    \cref{eq:xdynamicsepsilon} implies  $x(\theta)=x(0)$. 
 Both the functions $\phi^{-1}$ and $\sin s$ are  symmetric with respect to the origin.  This means that $Q(x(0),0)=0$ if and only if $x(0)=0$. 
Furthermore $(x(0),0)$ is regular  because 
\be \label{eq:zerosQ} \partial_{x}Q(0,0)  =  - \frac{\gamma^2}{\mu_d} \int_0^{2\pi} \frac{1}{\phi'(-\sin s/\mu_d)}\, ds<0\ee
and $\phi'(\hat{y})$ is always positive in $S_a$, since $\hat{y}\in]-\delta,\delta[$. Then the Implicit Function Theorem guarantees that for $0<\vare\ll1$ there exists   $x(0) = m(\vare)$ such that $Q(m(\vare),\vare)=0$. Hence $x(0) = m(\vare)$ belongs to a stable periodic orbit  since from \cref{eq:zerosQ} it follows that $\lvert \partial_{x(0)}P(x(0),\vare)\rvert<1$ for $0<\vare\ll1$.
 \end{proof}
Therefore, when $\mu_s>1$ the family of circles in $\Sigma_s$ bifurcates into a single attracting limit cycle on the slow manifold $S_{a,\vare}$.   This result gives an upper bound of the time $T$ of \cref{prop:regular_stiction_solution} as a function of $\vare$, since on the slow manifold $S_{a,\vare}$, after a time $t= \mathcal{O}(1/\vare)$, orbits are $\mathcal{O}(1)$ distant to the original   family of circles in $\Sigma_s$. Furthermore, the regularization of regular stiction solutions does not necessarily remain  uniformly close.  \\
It is not possible to make a statement similar to \cref{prop:regular_stiction_solution} for singular stiction solutions, as they have non-unique forward solutions at the singular level. A further understanding can be obtained by studying the reduced problem \cref{eq:reduced_problem}. This   differential algebraic equation,  is rewritten as a standard ODE by explicating the algebraic condition with respect to $x$ and  by differentiating it with respect to the time $t$:
 \be \label{eq:reducedproblem_originaltime} 
\begin{aligned}
 - \mu_d \phi'(\hat{y}) \hat{y}' &=  \cos \theta,\\
\theta' &=1.
\end{aligned}\ee
\begin{proposition}\label{prop:reduced_problem}
The circles $f^\pm \subset \{ \phi'(\hat{y}) =0\}$ are lines of singularities for the reduced problem \cref{eq:reducedproblem_originaltime}, and solutions reach them in finite time. On $f^\pm$, the points $(\hat{y},\theta) = (-\delta, \pi/2)$ and $(\hat{y},\theta) = (\delta, 3\pi/2)$ are folded saddles, while   $(\hat{y},\theta) = (\delta, \pi/2)$ and $(\hat{y},\theta) = (-\delta, 3\pi/2)$ are folded centers.   Moreover the intervals $\hat{I}^\pm\subset f^\pm$  defined as
\[\begin{aligned}
\hat{I}^- :=& \{ (x,\hat{y},\theta) \in \mathbb{R}^2\times\mathbb{T}^1 \mid \xi = \mu_s, \quad \hat{y}= - \delta,  \, \theta   \in]\pi/2,3\pi/2[\,\, \},\\
\hat{I}^+ :=& \{ (x,\hat{y},\theta) \in \mathbb{R}^2\times\mathbb{T}^1\, \mid \xi = -\mu_s, \hat{y}=  \delta,  \quad\theta\in [0,\pi/2[ \,\cup\,]3\pi/2,2\pi[\,\,\}, 
\end{aligned}
\]
have non-unique forward solutions.\end{proposition}
\begin{proof}
The time transformation $ \mu_d  \phi'(\hat{y})  d\hat{t} = dt$
allows to rewrite  system \cref{eq:reducedproblem_originaltime} as the {\it desingularized problem}
\be \label{eq:reducedproblem_short} 
\begin{aligned}
\dot{\hat{y}} &=  -\cos \theta,\\
\dot{\theta} &= \mu_d  \phi'(\hat{y}) ,
\end{aligned}\ee
in the new time $\hat{t}$. The difference between system  \cref{eq:reducedproblem_originaltime}  and \cref{eq:reducedproblem_short} is that $\hat{t}$ reverses the direction of time  within $C_r^\pm$.
Problem \cref{eq:reducedproblem_short} has four fixed points in $\mathbb{R}^2\times\mathbb{T}^1$.  The  points $(\delta, 3\pi/2)$ and $(-\delta,\pi/2)$ are hyperbolic saddles with eigenvalues $\pm\sqrt{\mu_d\lvert  \phi'' (\delta)  \rvert  }$, and eigenvectors respectively $[1, \mp\sqrt{\mu_d\lvert  \phi'' (\delta)  \rvert }]^T$ and $[1,\pm \sqrt{\mu_d\lvert  \phi'' (\delta)  \rvert }]^T$.  The remaining points $(\delta,\pi/2)$ and $(-\delta,3\pi/2)$ are centers  with eigenvalues $\pm i \sqrt{\mu_d\lvert  \phi'' (\delta)  \rvert }$, and eigenvectors $[1,\pm i\sqrt{\mu_d\lvert  \phi'' (\delta)  \rvert } ]^T$  and $[1,\mp i \sqrt{\mu_d\lvert  \phi'' (\delta)  \rvert } ]^T$ respectively. 
The inversion of  the time  direction on $C_r^\pm$ gives the dynamics of the reduced problem \cref{eq:reducedproblem_originaltime}. Thus a saddle in \cref{eq:reducedproblem_short} is a folded saddle in \cref{eq:reducedproblem_originaltime}, similarly for the centers. Also, $f^\pm$ become lines of singularities with the time inversion, and the segments  $\hat{I}^\pm$ have forward trajectories pointing inside both $C_a$ and  $C_r^\pm$, compare with \cref{fig:reduced_canard}. Since  $\theta'=1$, orbits reach or leave  $f^\pm$  in finite time.  \end{proof}
 \begin{figure}[t!]
\centering
\begin{subfigure}{.48\textwidth}\caption{ }\label{fig:reduced_canard}
\centering
\includegraphics[width=0.85\textwidth]{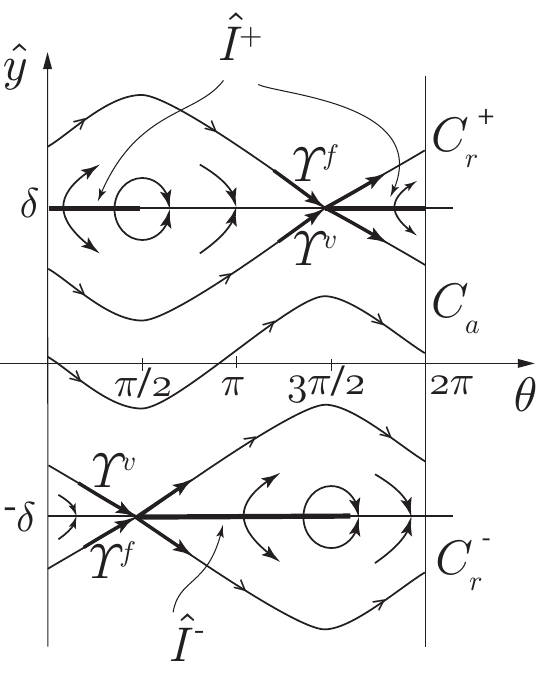}\end{subfigure}
\begin{subfigure}{.48\textwidth}\caption{ }\label{fig:repelling_slow_manifold}
\centering
\includegraphics[width=\textwidth]{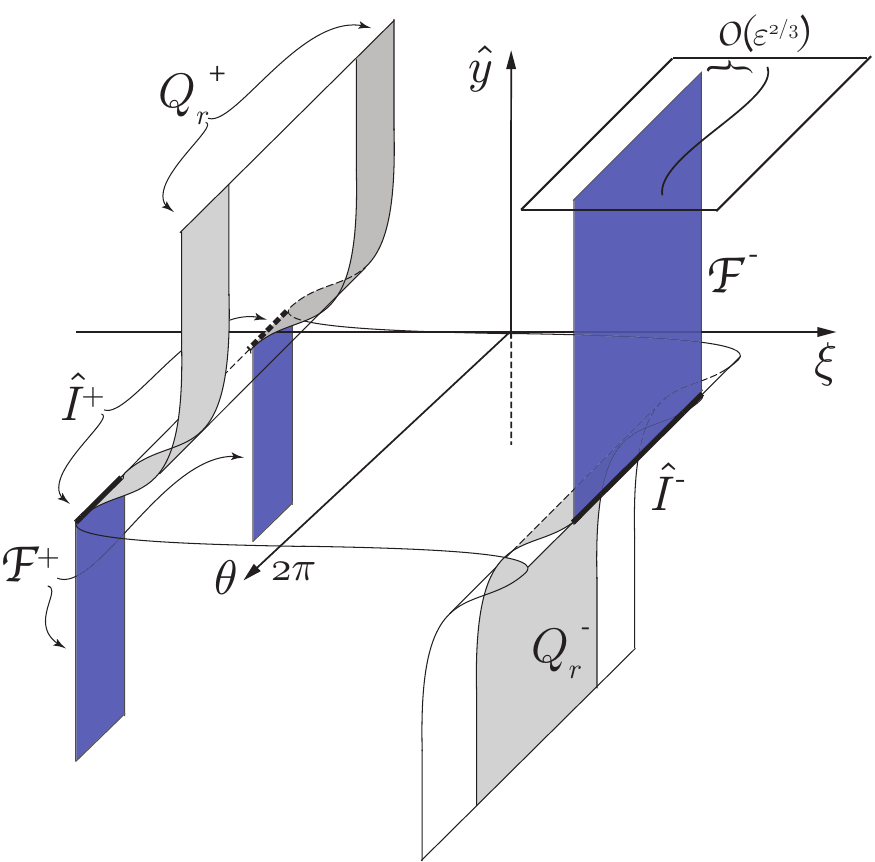}\end{subfigure}
\caption{\subref{fig:reduced_canard}: Phase space of the reduced problem \cref{eq:reducedproblem_originaltime}.   \subref{fig:repelling_slow_manifold}: repelling invariant manifolds $Q_r^\pm$ in grey, and foliations $\mathcal{F}^\pm$ in blue. 
}\label{fig:regularized}
 \end{figure}
\Cref{fig:regularized} illustrates the results of \cref{prop:reduced_problem}. In the $(x,y,\theta)$ coordinates, the segments $\hat{I}^\pm$ collapse onto the lines of non-uniqueness $I^\pm$ for $\vare=0$.  The   layer   problem \cref{eq:layer_problem}  adds a  further forward solution  in $\hat{I}^\pm$, since orbits may also leave a point of these lines by following a fast fiber for $\hat{y}\gtrless0$.   \\Each folded saddle has two special solutions: the {\it singular vrai canard} $\Upsilon^v$ that connects  $C_a$ to $C_r^\pm$, and the {\it singular faux canard} $\Upsilon^f$ that  does the opposite \cite{benoit1981a,dumortier1996a}. The vrai canard separates two different types of forward dynamics:  on one side of $\Upsilon^v$  orbits {\it turn}, that means  they remain on $C_a$. On the other side of $\Upsilon^v$ orbits  reach $f^\pm \setminus \hat{I}^\pm$ and then  {\it jump}, that is, they  move away from $C_0$ by  following  a fast fiber. Each singular canard is a periodic orbit that visits both $C_a$ and $C_r^\pm$, see \cref{fig:reduced_canard}.  The folded centers have no canard solutions \cite{krupa2010a} and for this reason they are not interesting for the analysis.
 Canards are a generic feature of systems with two slow and one fast variable. They appear for instance in    the Van der Pol oscillator \cite{guckenheimer2005a,vo2015a}, in a model for global warming \cite{wieczorek2011a} and in a model for transonic wind \cite{carter2017}. \\
When $0<\vare\ll1$ the singular  vrai canard  $\Upsilon^v$  perturbs into a maximal canard \cite{szmolyan2001a}. This orbit corresponds to the  intersection of  $S_{a,\vare}$ with  $S_{r,\vare}^\pm$. Hence the maximal canard remains  $\mathcal{O}(\vare)$-close to  $S_{r}^\pm$ for a time $t=\mathcal{O}(1)$.
\begin{figure}[t!]
\centering
\begin{subfigure}{.4\textwidth}\caption{ }\label{fig:nonuniqueness_regularized}\centering
\includegraphics[width=0.75\textwidth]{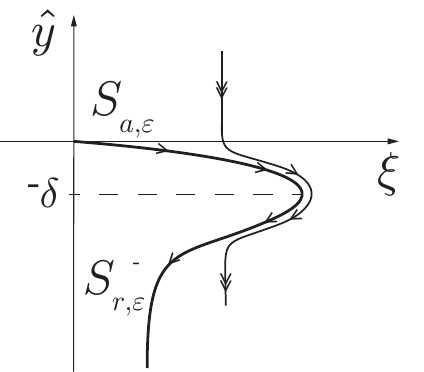}\end{subfigure}
\begin{subfigure}{.4\textwidth}\caption{ }\label{fig:nonuniqueness_perturbation}\centering
\includegraphics[width=0.75\textwidth]{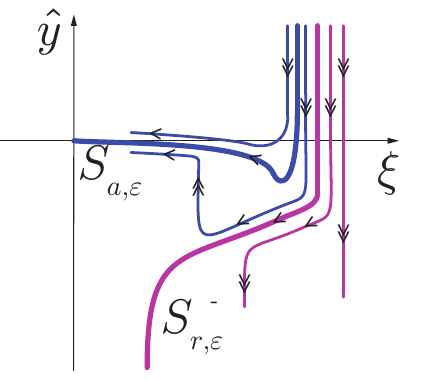}\end{subfigure}
\begin{subfigure}{.4\textwidth}\caption{ }\label{fig:canard_singular}\centering
\includegraphics[width=\textwidth]{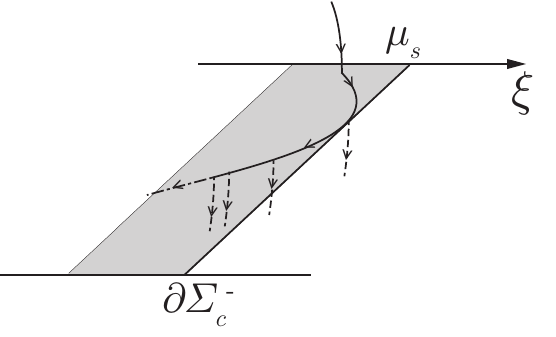}\end{subfigure}
\begin{subfigure}{.4\textwidth}\caption{ }\label{fig:nonuniqueness_singular}\centering
\includegraphics[width=\textwidth]{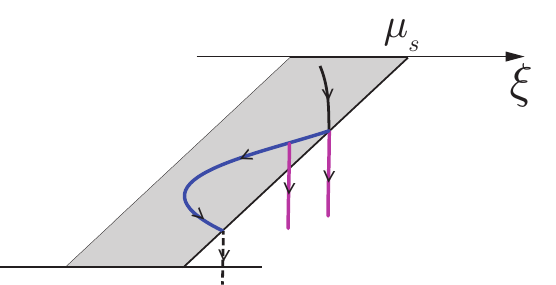}\end{subfigure}
\caption{\subref{fig:nonuniqueness_regularized}: A canard orbit at the intersection of $S_{a,\vare}$ with $S_{r,\vare}^-$. \subref{fig:nonuniqueness_perturbation}: Dynamics around a point of  $\hat{I}^-$, for  $0<\vare\ll1$. \subref{fig:canard_singular} and \subref{fig:nonuniqueness_singular}:  The same dynamics of \Cref{fig:nonuniqueness_regularized,fig:nonuniqueness_perturbation} in the $(x,y,\theta)$-coordinates.  The canard-like solutions leaving $\Sigma_{s,\text{stiction}}^-$  resemble  Carath\'eodory solutions of model \cref{eq:model_dimensionless}, compare with \cref{fig:artifacts_solutions}. }\label{fig:canard_limit}
 \end{figure}  
 Furthermore  a  family of orbits  remains exponentially close to the maximal canard for some time,  before being repelled  from $S_{r,\vare}^\pm$  \cite[p.~200]{kuehn2015a}. An orbit of this family is called a {\it canard} and \cref{fig:nonuniqueness_regularized} shows an example  of it. 
Define $Q_r^\pm$ as the subsets of  $C_r^\pm$ whose solutions, when flowed backwards in time,  intersect  the intervals of non-uniqueness $\hat{I}^\pm$. $Q_r^\pm$ are coloured in grey in \cref{fig:repelling_slow_manifold}. The lines $\hat{I}^\pm$ are, backwards in time, the base of a foliation of fast fibers $\mathcal{F}^\pm$, that are coloured in blue in \cref{fig:repelling_slow_manifold}. The following Proposition describes the role of  the repelling manifolds $Q_r^\pm$ for $0<\vare\ll1$. 
\begin{proposition}
For $0<\vare\ll1$  compact subsets  $S_{r}^\pm$ of $ Q_r^\pm$  perturb into the sets $S_{r,\vare}^\pm$ that are $\mathcal{O}(\vare)$-close to  $S_{r}^\pm$. The slow problem on $S_{r,\vare}^\pm$ is connected  backwards in time to a family of fast trajectories $\mathcal{F}_\vare^\pm$ that is $\mathcal{O}(\vare^{2/3})$-close to $\mathcal{F}^\pm$.    The orbits on $\mathcal{F}_\vare^\pm$ and $S_{r,\vare}^\pm$ separate  the trajectories that, after possibly having been exponentially close to $S_{r,\vare}^\pm$, are attracted  to the slow manifold $S_{a,\vare}$  to the ones that  follow a fast trajectory away from the slow surface.
\end{proposition}
\begin{proof}
By reversing the time orientation on the slow \cref{eq:slow_problem} and fast problem \cref{eq:fast_problem},  the orbits on $Q_r^\pm$ satisfy the assumptions of \cref{prop:regular_stiction_solution}. Hence the distance of $\mathcal{F}^\pm$ to $\mathcal{F}_\vare^\pm$ is  $\mathcal{O}(\vare^{2/3})$. Now consider again the true time direction, and take a set of initial conditions that is exponentially close to the fibers  $\mathcal{F}_\vare^\pm$. These orbits will follow the  repelling slow manifolds $S_{r,\vare}^\pm$ for a time $t= \mathcal{O}(1)$ \cite{szmolyan2001a}. The manifolds $S_{r,\vare}^\pm$ act as separators of two different futures: on one side the orbits will get attracted to the slow attracting manifold $S_{a,\vare}$, while on the other side they will jump away by following an escaping fast fiber, compare with \cref{fig:nonuniqueness_perturbation}. 
\end{proof}
It follows that around $\hat{I}^\pm$ and $\mathcal{F}^\pm$ there is  a high sensitivity to the initial conditions. Even though the $(x,\theta)$-dynamics on $C_{a}$ coincides with the one on $C_{r}^\pm$, trajectories close to  these two manifolds may have different futures. Orbits belonging to $S_{a,\vare}$ will  exit $S_{a,\vare}$  in a predictable point. On the other hand, the orbits that follow  $S_{r,\vare}^\pm$ are very sensitive, and may   escape from it at any time.  These two types of trajectories are coloured respectively in blue and magenta in \cref{fig:nonuniqueness_perturbation,fig:nonuniqueness_singular}.   The orbits that  follow  $S_{r,\vare}^\pm$ for some time are canard-like in the forward behaviour. However  in backward time they are connected to a family of fast fibers instead than to $S_{a,\vare}$ and for this reason they are not typical canards like $\Upsilon^{v}$.\\ 
In the  original coordinates $(x,y,\theta)$,   the canard trajectories of the folded saddles and  the canard-like solutions of the lines $\hat{I}^\pm$  leave the slow manifold in a point inside $\Sigma_{s,stiction}^\pm$, as in \cref{fig:canard_singular,fig:nonuniqueness_singular}.  In the piecewise smooth system these orbits satisfy the  Carath\'eodory condition \cref{eq:caratheodory} but they are not stiction solutions.  
It follows that some of the  Carath\'eodory solutions of \cref{eq:model_dimensionless} appear upon regularization of the stiction model: these are the trajectories of $Z_s$ that  intersect   $I^\pm$  backwards in time.  All the other Carath\'edory solutions of model \cref{eq:model_dimensionless} do not have a corresponding solution in the regularized model.
The interpretation of the solutions with canard is that   the slip onset is delayed with respect to the time when  the external forces have equalled the maximum static friction force. \Cref{fig:slip_stick_pos_projection} in \cref{sec:pos_regularized}, will show a numerical solution having this delay.

\section{Slip-stick periodic orbits}\label{sec:slipstickorbits} This section considers  a family of periodic orbits of model \cref{eq:model_dimensionless} that interacts with the lines of non-uniqueness $I^\pm$. Then \cref{sec:pos_regularized} discusses how the family perturbs in the regularized system \cref{eq:vectorfield_regularized} for $0<\vare\ll1$, by combining numerics and analysis. \\
Model \cref{eq:model_dimensionless} has several kinds of periodic motion: pure slip \cite{shaw1986a,csernak2006a}, pure stick \cite{hinrichs1998a}, non-symmetric slip-stick \cite{olsson2001a,galvanetto1999a,andreaus2001a,oestreich1996a,popp1990a}, symmetric slip-stick \cite{olsson2001a,hinrichs1998a}. 
This section focuses on the latter, as slip-stick orbits are likely to be affected by the non-uniqueness at $I^\pm$. \Cref{fig:phasespace_3D_slipstick} shows an example of such a trajectory.  The symmetric slip-stick trajectories can  be found by solving a system of algebraic equations, because   system \cref{eq:model_dimensionless}, in its non-autonomous form, is piecewise-linear in each region. Furthermore, it is sufficient to study only half the period, as ensured by the  following two \cref{lem:symmetry,lem:symmetric_orbit}.
\begin{lemma}\label{lem:symmetry}
System \cref{eq:model_dimensionless} has a symmetry
\be \label{eq:symmetry} S(x,y,\theta) = (-x, -y, \theta+\pi).\ee
\end{lemma}
\begin{proof} The map \cref{eq:symmetry} is a diffeomorphism $\mathbb{R}^2\times\mathbb{T}^1 \to \mathbb{R}^2\times\mathbb{T}^1$ that satisfies the condition for a symmetry $ Z(S(z)) = DS(z) Z(z)$, where $DS(z)$ is the Jacobian of $S(z)$ and $z = (x,y,\theta)$  \cite[p.~211]{meiss2007a}. 
\end{proof}
\begin{lemma}\label{lem:symmetric_orbit}
Let $\varphi_{t}(z)$ be the regular stiction orbit of system \cref{eq:model_dimensionless} at time $t$,  with initial condition $z = (x,y,\theta)$.  If $\varphi_{ \pi}(z) = (-x,-y,\theta+\pi)$ then the orbit is symmetric and periodic with period ${T} = 2\pi$.
\end{lemma}
\begin{proof}
Applying the symmetry map \cref{eq:symmetry}  to the point $\varphi_{ \pi}(z)$,  gives
\[ S(-x,-y,\theta+\pi) = (x, y, \theta+2\pi). \]
Since $Z(x, y, \theta+2\pi) \equiv Z(x, y, \theta)$ for any $\theta \in \mathbb{T}^1$,  the flow $\varphi_{t}(z)$ is symmetric and periodic, with symmetry \cref{eq:symmetry} and period ${T} = 2\pi$.
\end{proof}  
 \begin{figure}[t!]
 \centering
\includegraphics[width=0.7\textwidth]{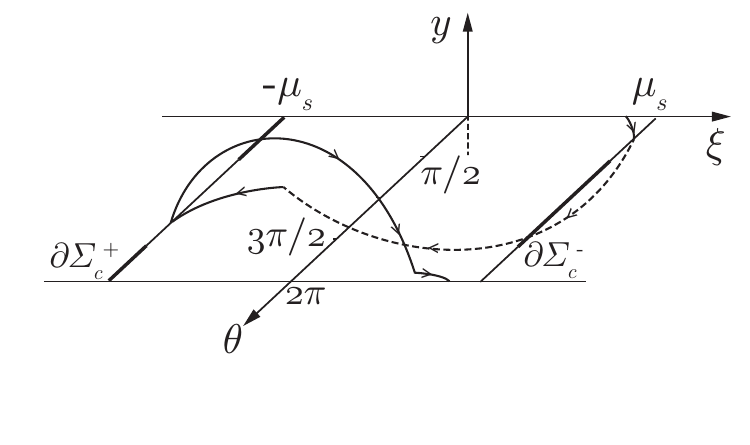}
\caption{A symmetric, slip-stick, periodic orbit with $\theta\in\mathbb{T}^1$. The dashed line represents trajectories in $Z^-$. The interest is to study how such orbit interacts with the intervals of non-uniqueness $I^\pm$ (in bold) under variation of a parameter. }\label{fig:phasespace_3D_slipstick}
 \end{figure}
 The  results of \cref{lem:symmetric_orbit} have been used in \cite{shaw1986a} even though the symmetry was not made explicit. Define $\varphi_{t}^\text{slip}(z_0)$ (resp. $\varphi_{t}^\text{stick}(z_1)$) the slip (stick) solution  of  $Z^-(z)$ ($Z_s(z)$) with initial conditions in $z_0$ ($z_1$).   The following Lemma states when these two solutions, pieced together, belong to a symmetric slip-stick periodic orbit.
 \begin{lemma}\label{lem:slipstick}
Necessary conditions for the slip and stick solutions $\varphi_{t}^\text{slip}(z_0)$ and  $\varphi_{t}^\text{stick}(z_1)$ to form the lower half of  a symmetric, slip-stick, periodic orbit are
\begin{subequations}\label{eq:slipstick_cond}
\begin{flalign}
\varphi_{\pi-\theta^*}^\text{slip}(z_0) &= \varphi_{0}^\text{stick}(z_1), \label{eq:goodconnection}\\
\varphi_{ \theta^*}^\text{stick}(z_1) &= S(z_0). \label{eq:goodsymmetry}
\end{flalign}\end{subequations}
where $ 0<\theta^*<\pi$  is the duration of one stick phase and $z_0 \in \partial \Sigma_c^-$, $z_1 \in \Sigma_s$.
\end{lemma}
Condition \cref{eq:goodconnection} guarantees the continuity between the stick and slip phase, while \cref{eq:goodsymmetry} guarantees the symmetry. The upper half-period of the orbit follows by applying the symmetry map \cref{eq:symmetry} to $\varphi_{t}^\text{slip}$ and $\varphi_{t}^\text{stick}$.
\begin{corollary}
Conditions \cref{eq:slipstick_cond} are equivalent to 
\begin{subequations} \label{eq:slipstick_equivalent}
\begin{flalign}
x^\text{slip}(\pi - \theta^*) &= -x_0, \label{eq:xislip_connection}\\
y^\text{slip}(\pi - \theta^*) &= 0, \label{eq:yslip_connection}\\
\pi-\theta^* &+\theta_0=  \theta_1 .
\end{flalign}\end{subequations}
Where $z_0 = (x_0,y_0,\theta_0)\in\Sigma_c^-$, $z_1 = (x_1,y_1,\theta_1)\in\Sigma_s$ and   $\varphi_{t}^\text{slip}(z_0)=(x(t),y(t),\theta(t))^{\text{slip}}$.
\end{corollary}
\begin{proof}
The stick solution of \cref{eq:model_dimensionless} with initial condition $z_1 = (x_1,0,\theta_1)$ is  $(x,y,\theta)^\text{stick}(t) = (x_1,0,t+\theta_1)$.   Condition \cref{eq:goodconnection} then implies that  $x^\text{slip}(\pi - \theta^*) = x_1$ and $y^\text{slip}(\pi - \theta^*) = 0$, while $\theta^\text{slip}(\pi - \theta^*) = \pi-\theta^* + \theta_0 = \theta_1$. Condition \cref{eq:goodsymmetry} adds furthermore that $x_1 = -x_0$. 
 \end{proof}
The stick-slip solutions of \cref{eq:model_dimensionless} are now investigated numerically. The system of  conditions \cref{eq:slipstick_equivalent} has five unknown parameters: $\gamma, \theta_0, \theta^*, \mu_s$ and $\mu_d$. It is reasonable to fix $\mu_s$ and $\mu_d$ as these are related to the material used, and then find a family of solutions of \cref{eq:slipstick_equivalent} by varying the frequency ratio $\gamma = \Omega/\omega$. The values used in the computations are listed in \cref{tab:1}.
   \begin{figure}[th!]
\centering
\begin{subfigure}{.4\textwidth}\caption{ }\label{fig:PWSfamily_gammalarge}
\includegraphics[width=\textwidth]{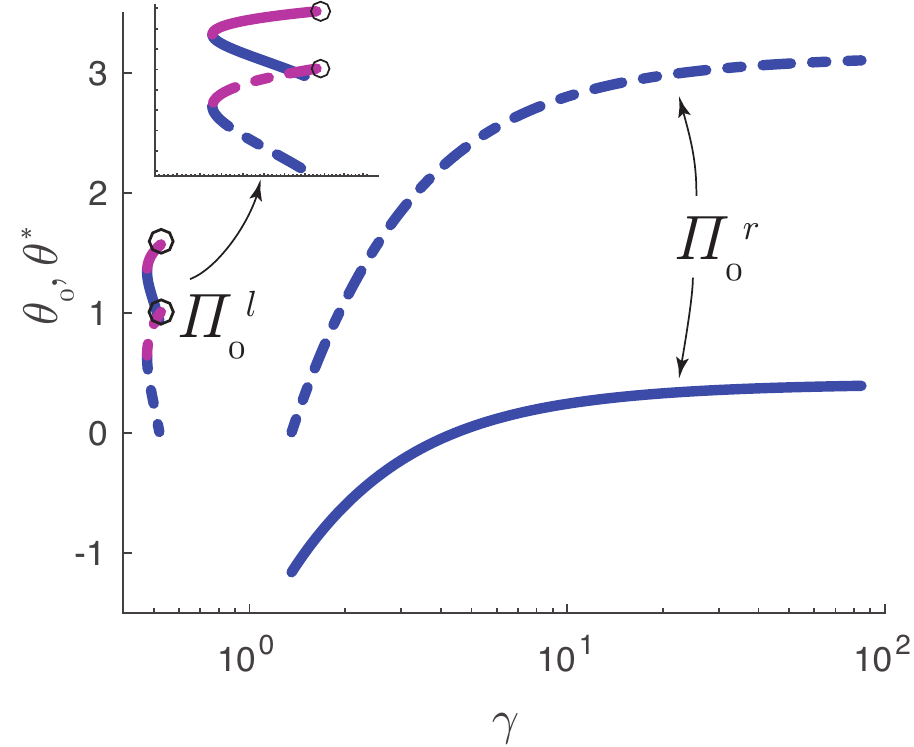}\end{subfigure}
\begin{subfigure}{.4\textwidth}\caption{ }\label{fig:Maxy_gammalarge}
\includegraphics[width=\textwidth]{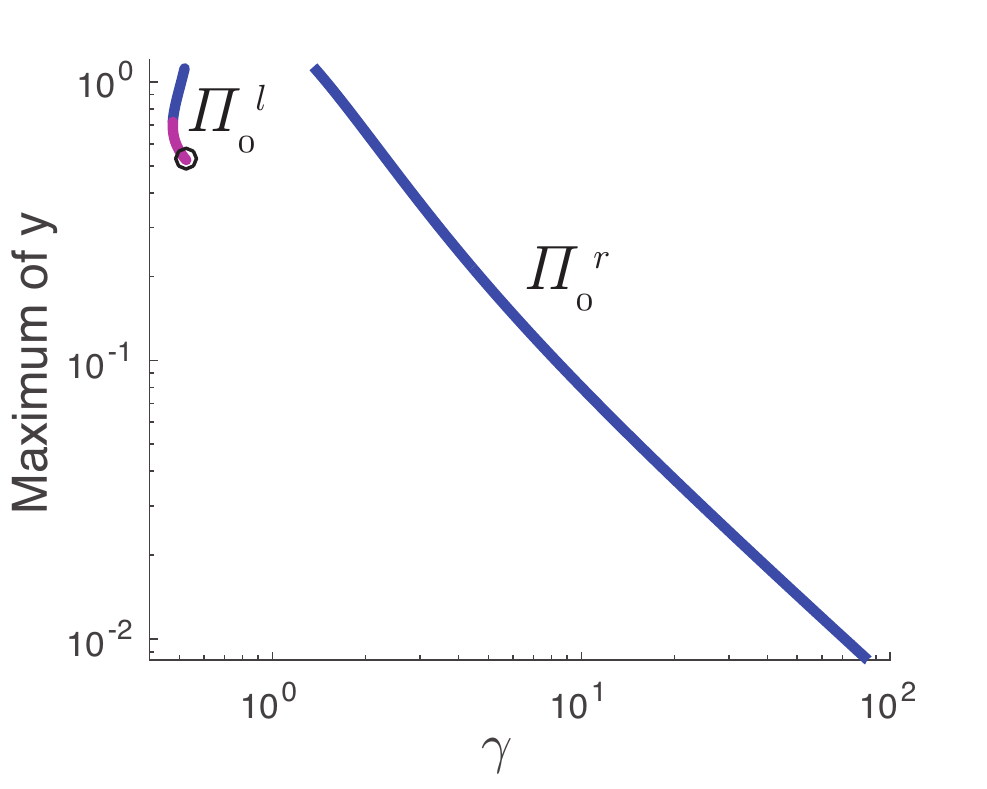}\end{subfigure}
\caption{\subref{fig:PWSfamily_gammalarge}: Two families of slip-stick orbits $\Pi_0^{l,r}$ of \cref{eq:model_dimensionless} for $\mu_s = 1.1$, $\mu_d = 0.4$. The solid line is $\theta_0$ while the dashed line is $\theta^*$. The blue denotes a stable periodic orbit, while the magenta a saddle periodic orbit. \subref{fig:Maxy_gammalarge}:  Maximum  amplitude of the orbits.
}\label{fig:PWSfamily}
 \end{figure}
\begin{table}
\caption{Parameters values used in the simulations.}
\label{tab:1}       
\centering
\begin{tabular}{cccc}
\hline\noalign{\smallskip}
$\mu_s$ & $\mu_d$ & $\vare$ & $\delta$  \\
\noalign{\smallskip}\hline\noalign{\smallskip}
1.1 & 0.4 & $10^{-3}$ & 0.6  \\
\noalign{\smallskip}\hline
\end{tabular}
\end{table}
Notice that conditions  \cref{eq:slipstick_equivalent} are  necessary  but not sufficient:  further admissibility conditions may be needed.  These are conditions that control that each piece of solution does not exit its region of definition, for example: the stick solution should not cross $\partial \Sigma_c^-$ before $t=\theta^*$, and should not cross $\partial \Sigma_c^+$ for any $t\in[0,\theta^*]$.
A numerical computation shows that system \cref{eq:slipstick_equivalent} has two branches of solutions $\Pi_0^{l,r}$, as shown in \cref{fig:PWSfamily}: one for $\gamma<1$ and one for $\gamma>1$. The branches are disconnected around the resonance for $\gamma=1$, where chaotic behaviour may appear \cite{andreaus2001a,csernak2006a,oestreich1996a}. 
The branch $\Pi_0^l$ for  $\gamma<1$ is bounded by pure slip orbits when $\theta^*\to0$, and by the visible tangency on $\Sigma_s$ when $\theta_0\to \pi/2$. The latter is marked with a circle in \cref{fig:PWSfamily_gammalarge}.  The  branch $\Pi_0^r$  for $\gamma>1$  is delimited by pure slip orbits when $\gamma\to 1$ since again $\theta^*\to0$, while when $\gamma\gg1$, that is the rigid body limit,  the family is bounded by $\theta^*\to\pi$. Here periodic orbits have a very short slip phase and an almost $\pi$-long stick phase. \\
A slip-stick orbit  of model \cref{eq:model_dimensionless} has three Floquet multipliers. Of these, one is trivially unitary, the second one is always zero and the last indicates the stability of the periodic orbit. The zero multiplier is due to the interaction of the periodic orbit with the sticking manifold $\Sigma_s$:  solutions lying on this surface are backwards non-unique. \Cref{fig:PWSfamily} denotes in blue the attracting periodic solutions  and in magenta the repelling ones. In particular the family $\Pi_o^l$ becomes    unstable  sufficiently close to the visible tangency at $\theta_0=\pi/2$, which is marked with a circle in \cref{fig:PWSfamily}. 
This is because the visible tangency    
 acts as a separatrix of two very different behaviours: on one side orbits jump, while on the other side they turn, recall \cref{fig:reduced_canard}.  


\subsection{Slip-stick periodic orbits in the regularized system}\label{sec:pos_regularized}
This section finds slip-stick periodic solutions of the regularized model \cref{eq:vectorfield_regularized} with a numerical continuation in AUTO \cite{Doedel2006}.  The solutions are then  compared   with the ones of the discontinuous system \cref{eq:model_dimensionless}. 
The regularization function  used  is a polynomial 
\[\phi(y) = y(ay^6 + by^4 + cy^2 + d),\]
within $y\in[-1,1]$, where the coefficients $a,b,c,d$ are determined by the conditions \cref{eq:regularization} for the parameters  listed in \cref{tab:1}.  Hence  $\phi(y)$ is $C^1$ for $y\in\mathbb{R}$.
\begin{figure}[h!]\centering
\begin{subfigure}{.48\textwidth}\caption{ }\label{fig:AUTOfamily_Maxy}
\includegraphics[width=\textwidth]{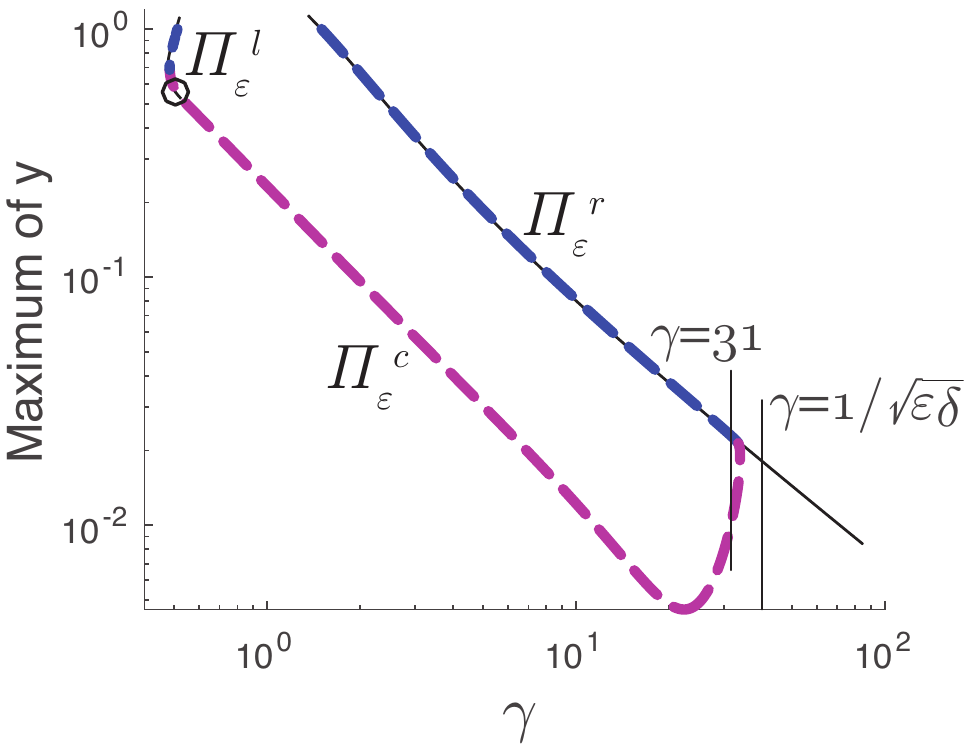}\end{subfigure}
  \begin{subfigure}{.48\textwidth}\caption{ }\centering\label{fig:slip_stick_pos_projection2}
\includegraphics[width=1\textwidth]{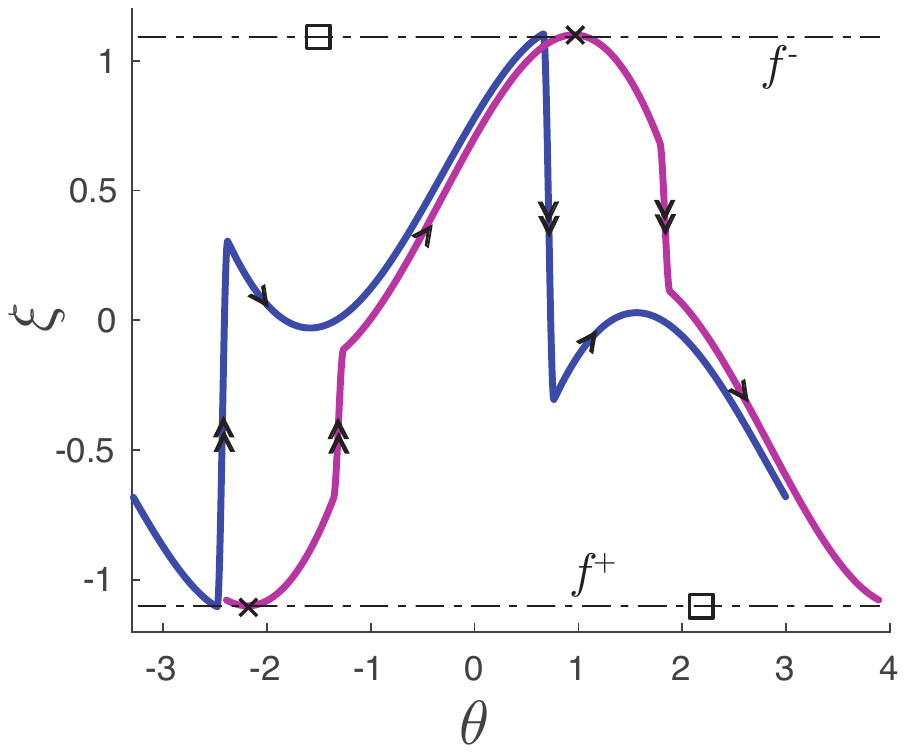}
 \end{subfigure}
  \begin{subfigure}{.48\textwidth}\caption{ }\centering\label{fig:slip_stick_pos_projection}
\includegraphics[width=1\textwidth]{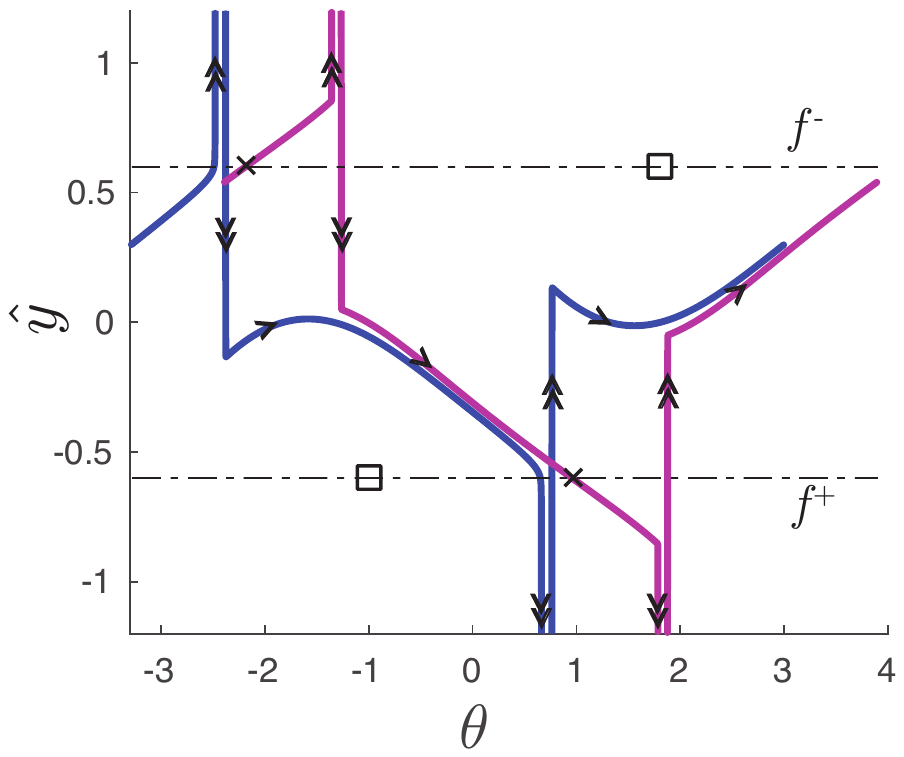}
 \end{subfigure}
 \begin{subfigure}{.48\textwidth}\caption{ }\centering\label{fig:slip_stick_pos_projection3}
\includegraphics[width=1\textwidth]{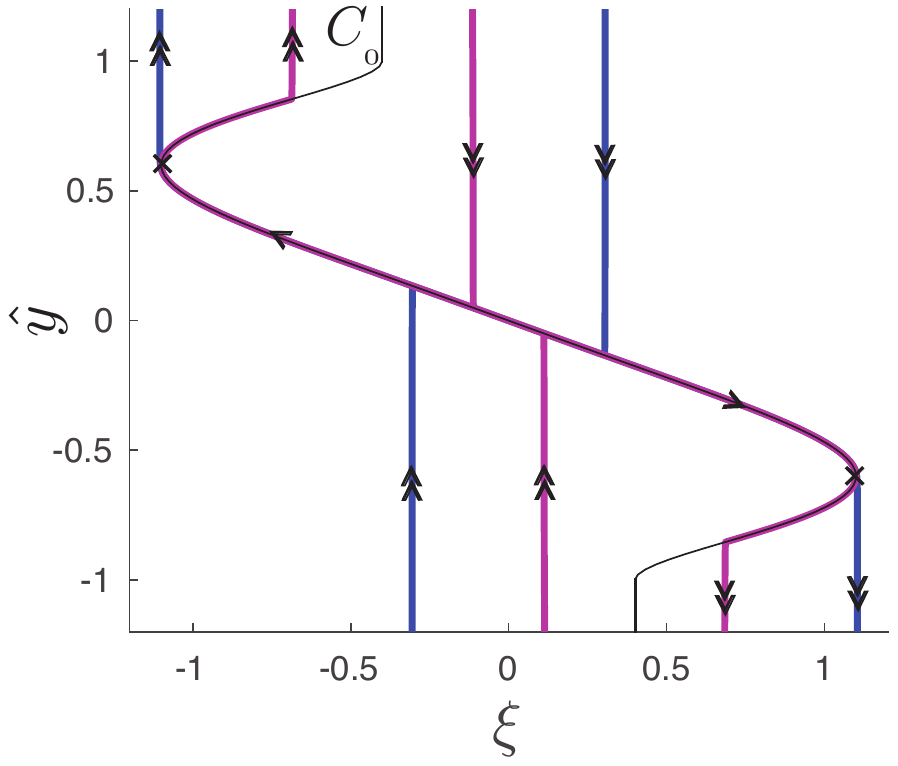}
 \end{subfigure}
\caption{Numerical simulation in AUTO.  \subref{fig:AUTOfamily_Maxy}:  In dashed the family   $\Pi_\vare$. The repelling branch $\Pi_\vare^c$ connects the two regular branches $\Pi_\vare^{l,r}$. Solid line: families $\Pi_0^{l,r}$. The colours denote the stability of the orbits, as in \cref{fig:PWSfamily}. 
\subref{fig:slip_stick_pos_projection2}: Two periodic orbits co-existing for $\gamma=31$: a regular slip-stick in blue and a slip-stick with canard segments in magenta. The {\tt x} marks the folded saddle while the $\square$ denotes the folded node. \subref{fig:slip_stick_pos_projection} and \subref{fig:slip_stick_pos_projection3}: Projections of \subref{fig:slip_stick_pos_projection2} in the $(\theta,\hat{y})$ and $(\xi,\hat{y})$-plane.  }\label{fig:AUTOsimuations}
\end{figure}
\Cref{fig:AUTOfamily_Maxy}  shows the family of slip-stick periodic orbits  $\Pi_\vare$ of system \cref{fig:AUTOfamily_Maxy}. This can be seen, loosely, as the union of three branches
\[ \Pi_\vare =    \Pi_\vare^l  \cup \Pi_\vare^c \cup \Pi_\vare^r ,\]
where $\Pi_\vare^{l,r}$ are  $\mathcal{O}(\vare^{2/3})$-close to the regular branches $\Pi_0^{l,r}$ \cite{szmolyan2004a}.  The branch $\Pi_\vare^c$  connects $\Pi_\vare^{l}$ to $\Pi_\vare^{r}$ at the rigid body limit, that is $\gamma\gg1$, and it consists   of slip-stick periodic orbits each having two canard segments. \Cref{fig:slip_stick_pos_projection2,fig:slip_stick_pos_projection,fig:slip_stick_pos_projection3} show for $\gamma=31$ two co-existing periodic orbits:  the magenta one belongs to $\Pi_\vare^c$  and the  blue one belongs to $\Pi_\vare^r$. In particular \cref{fig:slip_stick_pos_projection} shows the delay in the slip onset, when the orbit follows the canard, since the slip happens after a time $t= \mathcal{O}(1)$ with respect to when the orbit has intersected the fold lines $f^\pm$. \\
The existence of the branch $\Pi_\vare^c$ is supported by the next \cref{prop:existence_canard}. For this,   let $\Sigma_\text{out}$ be a cross-section orthogonal to the $y$-axis, so that   the fast fibers with base on the singular vrai canard  on $C_r^-$, intersect it on the line $L_\text{out,0}$. Furthermore, define $\Sigma_\text{in}$ the cross-section orthogonal to the $\xi$-axis so that it intersects $C_a$ on the line $L_\text{in,0}$, see \cref{fig:Poincare_transversality}.   
\begin{proposition}\label{prop:existence_canard}  Suppose that for $\vare=0$ there exists a smooth return mechanism $R : \Sigma_\text{out} \to \Sigma_\text{in} $ that maps $L_\text{out,0} \subset \Sigma_\text{out}$ onto $L_\text{in,0}\subset \Sigma_\text{in}$. Suppose furthermore that $L_\text{in,0}=R(L_\text{out,0})$ is transversal to the singular vrai canard $\Upsilon^v$.  Then for  $0<\vare\ll1$   there exists a unique, periodic orbit $\varphi_t^\vare(z)$ that has a canard segment,  and that tends to the singular canard for $\vare\to0$. Furthermore this orbit has a saddle stability with  Floquet multipliers: $\{1,\mathcal{O}(\text{e}^{-c_1/\vare}),\mathcal{O}(\text{e}^{\,c_2/\vare})\},$ with $c_{1,2}\in\mathbb{R}^+$. 
\end{proposition}
\begin{proof} First notice that for $0<\vare\ll1$ the singular vrai canard $\Upsilon^v$ on $C_r^-$ perturbs into a maximal canard that is $\mathcal{O}(\vare^{2/3})$-close to it. This maximal canard is the base of a foliation of fibers that intersect $\Sigma_\text{out}$ on a line $L_{\text{out},\vare}$ that is $\mathcal{O}(\vare^{2/3})$-close to $L_{\text{out},0}$. The return map $R(z)$ is smooth, so that $R(L_{\text{out},\vare})$ intersects $\Sigma_\text{in}$ in a line $L_{\text{in},\vare}$ that is $\mathcal{O}(\vare^{2/3})$-close to $L_{\text{in},0}$. The line $L_{\text{in},\vare}$ is  transversal to the maximal canard for $\vare$ sufficiently small, since  $L_{\text{in},0}$ was transversal to $\Upsilon^v$, and the perturbation is $\mathcal{O}(\vare^{2/3})$.\\
Now consider the backward flow of $L_{\text{out},\vare}$. This contracts to the maximal canard with an order $\mathcal{O}(\textnormal{e}^{-c/\vare})$. Hence it intersects $L_{\text{in},\vare}$ in an exponentially small set that is   centered around the maximal canard. This means that the reduced Poincar\'e map $P: L_{\text{in},\vare} \to L_{\text{in},\vare}$ is well defined and contractive in backwards time.  Hence it has a unique fixed point. Such fixed point corresponds to a periodic orbit with canard. It follows that the periodic orbit has an exponential contraction to the attracting slow manifold, and an exponential repulsion forward in time around the maximal canard. This determines the Floquet multipliers and consequently, the saddle stability.
\end{proof}
\begin{figure}[h!]
\begin{subfigure}{.45\textwidth}\caption{ }\label{fig:Poincare_transversality}
\includegraphics[width=\textwidth]{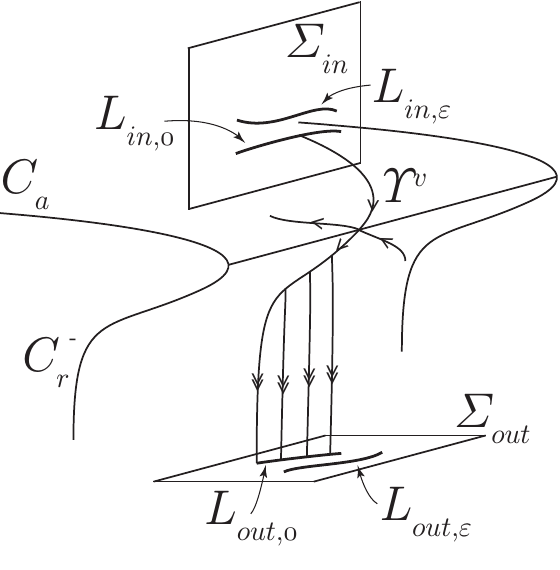}
\end{subfigure}
\begin{subfigure}{.45\textwidth}\caption{ }\centering\label{fig:Transversality_Canard}
\includegraphics[width=1\textwidth]{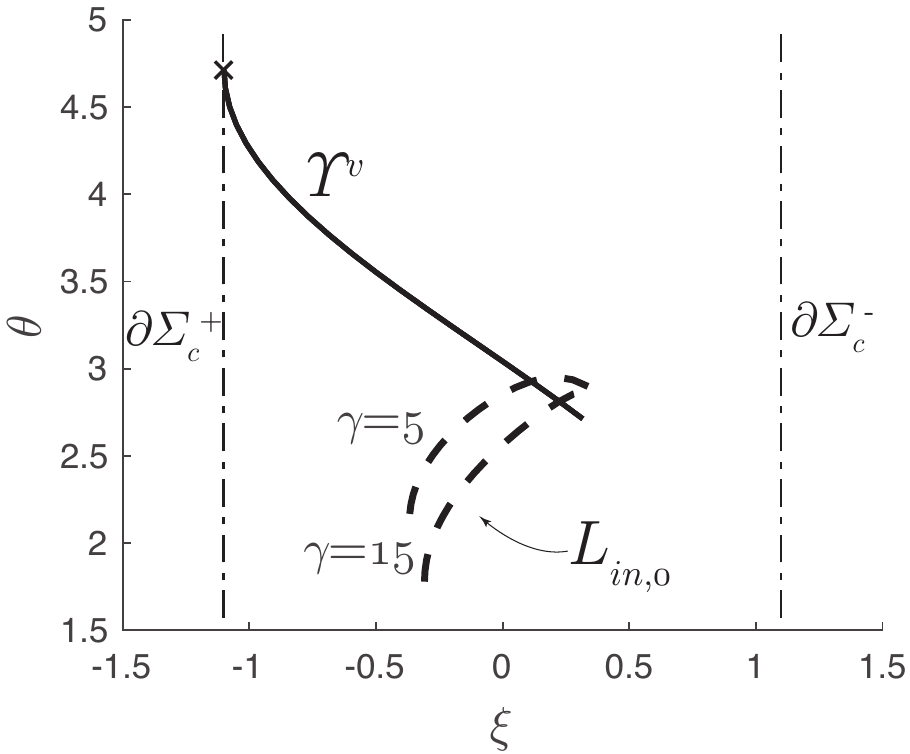}
 \end{subfigure}
\caption{\subref{fig:Poincare_transversality}: Construction of the cross-sections $\Sigma_{\text{in},\text{out}}$. \subref{fig:Transversality_Canard}: Numerical simulation showing that $R(L_{\text{out},0})$ (dashed line) is transversal to $\Upsilon^v$ (solid line) for $\vare=0$ and $\gamma=\{5,15\}$.   The visible tangency is marked with {\tt x}. The dashed-dotted lines are $\partial\Sigma_c^\pm$.    }\label{fig:Canard_pos_existence}
\end{figure}
\Cref{fig:Transversality_Canard}  shows numerically that the discontinuous model  \cref{eq:model_dimensionless} satisfies the assumptions of  \cref{prop:existence_canard}. This supports the existence of the branch $\Pi_\vare^c$ in the regularized model for $\vare$ sufficiently small. Because of the symmetry, the branch $\Pi_\vare^c$ has two canards segments for each  period. 
A {\it canard explosion} may appear when a family of periodic orbits interacts with  a canard. The explosion is defined as the transition from a small oscillation to a relaxation oscillation for an exponentially small   variation in the parameter \cite{krupa2001c}.  However system  \cref{eq:vectorfield_regularized} has no canard explosion: \cref{fig:AUTOfamily_Maxy}  shows that the maximum amplitude of the oscillations does not  increase  with the continuation from $\Pi_\vare^l$ to $\Pi_\vare^c$.  The effect of the canard is instead in the explosion of one of the Floquet multipliers as previously stated in \cref{prop:existence_canard}, and observed numerically in AUTO. 
The saddle stability of the family $\Pi_\vare^c$ implies that the periodic orbits of $\Pi_\vare^c$ are always repelling, even with a time  inversion. Hence these periodic  orbits are not visible in standard simulations. It could be interesting to make an experiment, with very high precision in the initial conditions, where the effects of the canard are measurable. If canard solutions appear, then this would support the validity of the stiction model and of its regularization. 
\begin{proposition}\label{prop:large_gamma}
The branch $\Pi_\vare^c$ is bounded above by  $\gamma = 1/\sqrt{\vare\delta}$ for $0<\vare\ll1$.
\end{proposition}
\begin{proof}
Differentiate $\xi(x,\theta) = \gamma^2 x + \sin(\theta)$ with respect to time, and rewrite the slow problem \cref{eq:slow_problem} in the $(\xi,\hat{y},\theta)$ variables
\[\begin{aligned}
  \xi' &= \gamma^2 \vare \hat{y} + \cos\theta,\\
\vare \hat{y}' &= - \xi - \mu_d \phi(\hat{y}),\\
\theta' &= 1.
\end{aligned} \]
If $\gamma^2 = \mathcal{O}(1/\vare)$, it makes sense to introduce the rescaling $\Gamma:= \gamma^2\vare$, so that the slow problem  becomes
\[\begin{aligned}
  \xi' &= \Gamma \hat{y} + \cos\theta,\\
\vare \hat{y}' &= - \xi - \mu_d \phi(\hat{y}),\\
\theta' &= 1.
\end{aligned} \]
This system has again a multiple time-scale with critical manifold  \cref{eq:critical_manifold}. Its  reduced problem in the  time $\hat{t}$ is
\be\label{eq:reducedproblem_gammalarge}\begin{aligned}
\dot{\hat{y}} &=  -\Gamma \hat{y} - \cos \theta,\\
\dot{\theta} &=  \mu_d  \phi'(\hat{y}) . 
\end{aligned}\ee
Notice that \cref{eq:reducedproblem_gammalarge} differs from the desingularized problem \cref{eq:reducedproblem_short} only for the term $\Gamma\hat{y}$ in the $\hat{y}$ dynamics. 
The fixed points of \cref{eq:reducedproblem_gammalarge} exist if $\lvert \Gamma\delta \rvert \leq1$ and they have coordinates $\hat{y}=\pm \delta, \, \cos\theta = \mp \Gamma\delta$. The comparison of system \cref{eq:reducedproblem_gammalarge} with the desingularized problem \cref{eq:reducedproblem_short} shows that the fixed points  have shifted along the $\theta$-direction. In particular the saddles have moved backwards while the centers have moved forward. Furthermore the centers have become stable foci. For increasing values of $\Gamma$ the stable  foci  turn into stable nodes. When $\lvert \Gamma\delta \rvert=1$ pairs of  saddles and  nodes collide and  disappear through a saddle-node bifurcation of type I \cite[Lemma 8.5.7]{kuehn2015a}.   Beyond this value canard solutions cease to exist. Such a condition is equivalent to $\gamma = 1/\sqrt{\vare\delta}$.
\end{proof}
The bound  $\gamma=1/\sqrt{\vare\delta}$, that is highlighted in \cref{fig:slip_stick_pos_projection2}, is larger than the  value of $\gamma$ for which the family $\Pi_\vare^c$ folds. In particular, at the turning point,   the folded foci have not turned  into folded nodes yet. Thus the collision of the folded saddles with the folded foci   is not a direct cause of the saddle-node bifurcation of $\Pi_\vare^c$, but gives only an upper bound for the existence of the family. When the folded nodes appear, there might exist further periodic orbits that exit the slow regime through the canard associated to the stable nodes.  \\   
Furthermore, the orbits of $\Pi_\vare^c$  interact  with the folded saddle only, but they do not interact with the other points  of  $\hat{I}^\pm$. 
The regularized problem \cref{eq:vectorfield_regularized} may have other families of periodic orbits that interact with $\hat{I}^\pm$. For  example, a family of pure slip periodic orbits, that reaches $\hat{I}^\pm$ from a fast fiber and then jumps off through a canard-like solution. However, this family would also turn unstable when passing sufficiently close to the canards,  because of the high sensitivity to the initial conditions around  $\mathcal{F}^\pm$. In particular an  explosion in the Floquet multipliers is again expected, because of   \cref{prop:existence_canard}.  
\section{Conclusions}\label{sec:conclusion}
Stiction is a widely used formulation of the friction force, because of its simplicity. However this friction law  has issues of non-uniqueness at the slip onset, that in this manuscript are highlighted in a  friction oscillator model. This model is a discontinuous, non-Filippov system, with subregions having a non-unique forward flow. The forward non-uniqueness is problematic in numerical simulations: here a choice is required and hence valid solutions may be discarded. A regularization of the model resolves the non-uniqueness by finding a repelling slow manifold that separates forward sticking  to forward slipping solutions. Around the slow manifold there is a high sensitivity to the initial conditions. Some trajectories remain close to this slow manifold for some time before being repelled. These trajectories, that mathematically are known as canards, have the physical interpretation of delaying the slip onset  when the external forces have equalled the maximum static friction force at stick. This result could potentially be verified experimentally, thus furthering the understanding of friction-related phenomena. Indeed the appearance of the canard solutions is a feature of stiction friction rather than of the specific friction oscillator model. For example the addition of a damping term on the friction oscillator, or the problem of a mass on an oscillating belt would give rise to similar canard solutions. \\
The canard solutions of the regularized systems can be interpreted, in the discontinuous model, as Carath\'eodory trajectories  that allow the slip onset in points inside the sticking region. These Carath\'eodory orbits are identified by being  backwards transverse to the lines of non-uniqueness. \\
The manuscript shows also that the  regularized system has a family of periodic orbits  $\Pi_\vare$ interacting with the folded saddles. The orbits with canard $\Pi_\vare^c \subset \Pi_\vare$ have a saddle stability, with   Floquet multipliers $\mathcal{O}(e^{ \pm c \vare^{-1}})$. Furthermore, the family $\Pi_\vare^c$ connects, at the rigid body limit, the two families of slip-stick periodic orbits $\Pi_0^{l,r}$ of the discontinuous problem. Further periodic orbits may interact with the canard segments. 

%
 
\section*{Acknowledgments}
The first author  thanks Thibault Putelat and Alessandro Colombo for the useful discussions. We acknowledge the Idella Foundation for supporting the research. This research was partially done whilst the first author was a visiting researcher at the Centre de Recerca Matem\`atica in the Intensive Research Program on Advances in Nonsmooth Dynamics.


\section*{References}
\bibliographystyle{siam}
\bibliography{bibliography/Stiction}                         
\end{document}